\documentclass{aptpub}

\authornames{KhudaBukhsh \emph{et al.}} 
\shorttitle{FCLT for SI Process on CM Graphs} 

\numberwithin{equation}{section}
\usepackage[T1]{fontenc}
\usepackage[utf8]{inputenc}

\usepackage{amsmath}

\usepackage{subcaption}
\usepackage{enumitem}

\usepackage{mathtools}

\usepackage{graphicx}
\graphicspath{{Figures/}}

\newtheorem{myTheorem}{Theorem}

\newtheorem{MyCorollary}{Corollary}

\newtheorem{myLemma}{Lemma}
\newtheorem{myRemark}{Remark}

\newcommand{\setOfReals}{\mathbb{R}}
\newcommand{\setOfNaturals}{\mathbb{N}}
\newcommand{\setOfNonnegativeIntegers}{\mathbb{N}_0}
\newcommand{\setOfPositiveReals}{\setOfReals_{+}}

\newcommand{\setTime}[1]{ \mathcal{T}_{#1}}

\newcommand{\borel}[1]{\mathcal{B} (#1 )}

\newcommand{\BigO}[1]{\mathop{}\!O{\left(#1\right)}}
\newcommand{\smallO}[1]{\mathop{}\!o{\left(#1\right)}}

\newcommand{\optionalVariation}[1]{ \left[ #1 \right] }
\newcommand{\predictableVariation}[1]{ \left\langle #1 \right\rangle }

\newcommand{\cardinality}[1]{\mid #1  \mid}

\newcommand{\cadlag}{c\`adl\`ag\,}

\newcommand{\Cov}[1]{\mathrm{Cov}(#1)}

\newcommand{\indicator}[1]{\mathsf{1}\left(#1\right)}
\newcommand{\FOperator}[2]{\mathbb{F}_{\mathrm{#2}} \left({#1} \right) }
\newcommand{\HOperator}[2]{\mathbb{H}_{\mathrm{#2}} \left( {#1} \right) }
\newcommand{\GOperator}[2]{\mathbb{G}_{\mathrm{#2}} \left( {#1} \right) }
\newcommand{\concentration}[1]{\left[ #1 \right]}

\newcommand{\diffOperator}[2]{\mathbb{D}^{#2} #1}

\newcommand{\differential}[1]{\, \mathrm{d} #1}

\newcommand{\norm}[1]{\left\lVert#1\right\rVert}

\newcommand{\E}{\mathbb{E}}

\newcommand{\Eof}[1]{\E\left[#1\right]}

\newcommand{\NPermuteR}[2]{ (  #1)_{#2}  }

\renewcommand{\prob}{\mathbb{P}}
\newcommand{\probOf}[1]{\prob\left(#1\right)}

\newcommand{\history}[1]{\mathcal{F}_{#1}  }

\newcommand{\eqstop}{.}
\newcommand{\eqcomma}{,}
\newcommand{\defeq}{\coloneqq}

\DeclareMathOperator{\ConvInProb}{\xrightarrow[]{ \hspace*{4pt}  \mathrm{         P   } }}
\DeclareMathOperator{\ConvAlmostSure}{\xrightarrow[]{ \hspace*{4pt}  \mathrm{         a.s.   } }}
\DeclareMathOperator{\ConvInDist}{  \overset{\hspace*{4pt}  \mathcal{D}  }{\implies } }

\DeclareMathOperator*{\Bigcdot}{\bullet}

\newcommand{\myExp}[1]{\exp \bigl( #1 \bigr)  }

\newcommand{\ie}{\textit{i.e.}}
\newcommand{\eg}{\textit{e.g.}}

\newcommand{\quotes}[1]{``#1''}

\usepackage{xcolor}
\usepackage{xkeyval}
\input{tudcolours.def}

\begin{document}
\title{A Functional Central Limit Theorem for SI  Processes on Configuration Model Graphs}

\authorone[University of Nottingham]{Wasiur~R.~KhudaBukhsh}
\addressone{School of Mathematical Sciences, The University of Nottingham, University Park, Nottingham, NG7 2RD, UK}

\authortwo[BHP Billiton]{Casper Woroszylo}
\addresstwo{480 Queen Street, Level 12,
	Brisbane QLD 4000 Australia
}
\authorthree[The Ohio State University]{Grzegorz A. REMPA{\L}A}
\addressthree{Mathematical Biosciences Institute,
				  The Ohio State University,
				    Jennings Hall 3rd Floor, 1735 Neil Ave.,
				      Columbus, OH 43210,
				      United States of America
}
\authorfour[Technische Universit\"at Darmstadt]{Heinz Koeppl}
\addressfour{Bioinspired Communication Systems,
		  	Technische Universit\"at Darmstadt,
		  	Rundeturmstrasse 12,
		  	64283 Darmstadt,
		  	Germany
        } 

\begin{abstract}
We study a stochastic compartmental susceptible-infected (SI) epidemic process on a configuration model (CM) random graph with a given degree distribution over a finite time interval. We split the population of graph vertices  into two compartments, namely, $\mathrm{S}$ and $\mathrm{I}$, denoting   susceptible and  infected vertices, respectively. 
In addition to the sizes of these two compartments, we  keep track of the counts of  $\mathrm{SI}$-edges (those connecting a susceptible and an infected  vertex), and $\mathrm{SS}$-edges (those connecting two susceptible vertices). We describe the dynamical process in terms of these counts and  present   a functional central limit theorem (FCLT) for them 
as the number of vertices in the random graph grows to infinity.  The FCLT  asserts  that the  counts, when appropriately scaled, converge weakly to a continuous Gaussian vector semimartingale process in 
the space of 
vector-valued \cadlag functions 
endowed with the Skorohod topology. 
We discuss applications of the FCLT in percolation theory  and in modeling spread of computer viruses. We also provide simulation results illustrating FCLT for some common degree distributions.
\end{abstract}




\section{Introduction}\label{sec:introduction}
Large-graph scaling limits, such as the functional laws of large numbers (FLLNs) and the functional central limit theorems (FCLTs),  of dynamical processes on random graphs have received much attention 
 of late. Although dynamical processes on random graphs themselves have long been studied by mathematicians, physicists, epidemiologists, computer scientists and engineers, a comprehensive and mathematically rigorous body of work on various scaling limits  under general settings remains elusive. Such scaling limits have been derived rigorously only for a handful of special cases to date. Notable breakthroughs in the context of epidemiological processes include \cite{hakan_andersson1998,BALL2002SIR,decreusefond2012large,janson2014law,jacobsen2016large}, appearing primarily in the probability theory literature. They provide functional laws of large numbers  under various sets of technical assumptions. However,  scaling limits in the form of  functional central limit theorems   have not been well investigated to the best of our knowledge. 
 We are not aware of any rigorously derived FCLT for dynamical processes on random graphs, except for diffusion-type approximations attempted in special cases like an early-stage epidemic in \cite{House2014dynamics} and normal approximation theorems for graph statistics in configuration model random graphs \cite{barbour2019CLT}. In this paper, we provide an FCLT for a particular type of binary dynamics
on configuration model (CM) random graphs (see \cite[Chapter 7]{hofstad2017randomVol1}, \cite{bollobas1998random}) as $n$, the number of vertices in the graph, grows to infinity.

\subsection{Our Contribution}
In the current paper we study a stochastic compartmental susceptible-infected (SI) epidemic process on a configuration model random graph with a given degree distribution over a finite time interval $\setTime{0} \defeq [0,T],$ 
for some $ T>0$. In this setting, we segregate the population into two compartments, namely, $\mathrm{S}$ and $\mathrm{I}$, containing  the susceptible and the infected individuals, respectively.
 In addition to the sizes of these two compartments, we  keep track of the counts of  $\mathrm{SI}$-edges (those connecting a susceptible and an infected  individual) and $\mathrm{SS}$-edges (those connecting two susceptible individuals). We describe the dynamical process in terms of these counts and  present   a functional central limit theorem  for them as $n$ 
 grows to infinity. To be precise, let $ X_{\mathrm{SI},i} (t)$ and $X_{\mathrm{SS},i}(t)$ denote the numbers of infected  and susceptible neighbours of a susceptible vertex~$i$ at time~$t$. Based on these local processes, define
 \begin{align}
 X_{\mathrm{SS}}(t) \defeq \sum_{i \in S(t)} X_{\mathrm{SS},i}(t) \quad \text{ and } X_{\mathrm{SI}}(t) \defeq \sum_{i \in S(t)} X_{\mathrm{SI},i}(t) \eqcomma \nonumber
 \end{align}
 where $S(t)$ denotes the set of susceptible vertices at time $t$. We will denote the set of infected vertices at time $t$ by $I(t)$.  While $X_{\mathrm{SI}}$ is the count of $\mathrm{SI}$-type edges, the process $X_{\mathrm{SS}}$ counts each $\mathrm{SS}$-type edge twice.  Let
 \[ X(t) \defeq ( X_{\mathrm{S}}(t),  X_{\mathrm{SI}}(t), X_{\mathrm{SS}}(t) )\]  denote the aggregated state vector of the system at time $t \ge 0$, where 
 $X_{\mathrm{S}} (t) \defeq \cardinality{ S(t)}$ keeps track of the number of susceptible vertices. Also, let   $$X_{\mathrm{S}\Bigcdot}(t)\defeq X_{\mathrm{SI}}(t)+X_{\mathrm{SS}}(t). $$ 
 A functional law of large numbers for the SI process approximates the scaled counts  $n^{-1} X$ by the solution to a system of Ordinary Differential Equations (ODEs). That is, $n^{-1} X \approx x$ for sufficiently large $n$, where  $x$ is the solution to the FLLN limit ODE. In this paper, we prove an FCLT for the fluctuations of the scaled process $n^{-1} X$ around the FLLN limit $x$. Our main result (Theorem~2 in Section~\ref{clt})   asserts that, under appropriate technical assumptions, the fluctuation process $Y(t)\defeq \sqrt[]{n} \left( n^{-1} X(t) - x(t)  \right)$  converges weakly to a Gaussian vector semimartingale in $D^{(3)}$, the  space of 
real 3-dimensional vector-valued \cadlag functions on $\setTime{0}$  endowed with the Skorohod topology.

The technical assumptions needed for the result to hold will be made precise later in Section~\ref{sec:LLN}. A  precise statement of the FCLT  result with additional discussion and further details is presented in Section~\ref{sec:FCLT}.

\subsection{Proof Strategy and Paper Outline}
Our derivation relies on the application of an FCLT for local martingales due to Rebolledo, referred to   as the \emph{Rebolledo theorem} hereinafter. In \cite{rebolledo1980central}, Rebolledo  
 provided sufficient conditions for the convergence of local martingales to a continuous Gaussian (vector) martingale
in terms of 
the associated optional and predictable quadratic variation processes, and the martingale process containing \quotes{big} jumps of the original process. 
%
Helland later provided a simpler proof of the {Rebolledo theorem} in \cite{helland1982central}. 
 For our purpose, we do not need the Rebolledo  theorem in its full generality;  a version of it tailored to the setting of  square integrable martingales suffices and,  for the sake of completeness,  we provide the statement of such a version in Appendix~\ref{sec:rebolledo}.

The first step towards the FCLT is to perform a Doob-Meyer decomposition (\cite{meyer1962decomposition}) of the semimartingale of the vector of counts into a zero-mean martingale and a compensator.
 Then one can show that the predictable quadratic variation of the appropriately scaled martingale process 
  converges in probability to a deterministic quantity. Furthermore, in the limit, its sample paths turn out to be close to continuous  in the sense that  their ``big'' jumps vanish. Weak convergence in the sense of \cite{billingsley2013convergence} is then established by applying  the Rebolledo  theorem.

The rest of the paper is structured as follows:  The construction of the CM random graph and the epidemic process on it are described in detail in Section~\ref{sec:SIonCM}. In Section~\ref{sec:LLN}, we make our technical assumptions precise and provide a law of large numbers before presenting our FCLT (Theorem~\ref{thm:FCLT})  in Section~\ref{sec:FCLT}. Necessary technical lemmas are also discussed leading up to the FCLT.  
In Section~\ref{sec:applications}, we discuss applications of our FCLT in percolation theory (from a non-equilibrium statistical mechanics point of view),   and in computer science in the context of spread of computer viruses. In Section~\ref{sec:applications}, we also discuss the peculiar case of Poisson-type degree distributions under which the correlation equations-based mean-field pair approximation approach  correctly estimates  the limiting process.  For illustration,  we  provide simulation results for some common degree distributions.

\section{The SI process on CM random graphs}
\label{sec:SIonCM}
\subsection{Notational Conventions}

We use the notations $\setOfNaturals$ and $\setOfReals$ to denote the set of natural numbers and the set of real numbers respectively. Also, we use $\setOfNonnegativeIntegers \defeq  \setOfNaturals \cup \{0\}$ and $\setOfPositiveReals \defeq  \setOfReals \setminus (-\infty,0]$. Given $R \subseteq \setOfReals$, we denote the $\sigma$-field of {Borel} subsets of~$R$ by $\borel{R}$. 
Recall $\setTime{0} \defeq [0, T]$ for some $t>0$. We denote by $D=D(\setTime{0})$ the space of real functions on $\setTime{0}$ that are right continuous and have left-hand limits. Functions in $D$ are called \cadlag.
Unless otherwise mentioned, the space~$D$ is assumed endowed with the Skorohod topology \cite[Chapter 3]{billingsley2013convergence}, which turns $D$ into a Polish space. We call $D$ the Skorohod space. Let the triplet $(\Omega, \history{}, \prob)$ denote our probability space. For an event~$A$, we use $\indicator{A}$ to denote the indicator (or characteristic) function of $A$. We shall use the following shorthand notation $\NPermuteR{a}{b}= a(a-1)(a-2)\cdots(a-b+1)$ for $a>b$ and $a,b \in \setOfNaturals$. The symbols $\BigO{.} $ and $\smallO{.}$ are the big~O and small~o notations respectively, and they carry their usual meanings. For a differentiable function $f   $ defined on some set $E \subseteq \setOfReals^d$, we denote its partial derivative with respect to the $i$-th variable by   $\partial_i f$, for $i= 1,2, \ldots, d$. With some abuse of notation, we use $\partial f(x)$ to denote the derivative of a differentiable function of a single variable at~$x$. 
For a stochastic process $Z(t)$ with paths in $D$, we denote the associated  jump  process $Z(t)-Z(t-)$ by $\delta Z(t)$.  
We also denote  $\setTime{} \defeq (0,T] \subset \setTime{0}$.

\begin{figure}
	\centering
	\includegraphics[width=\textwidth]{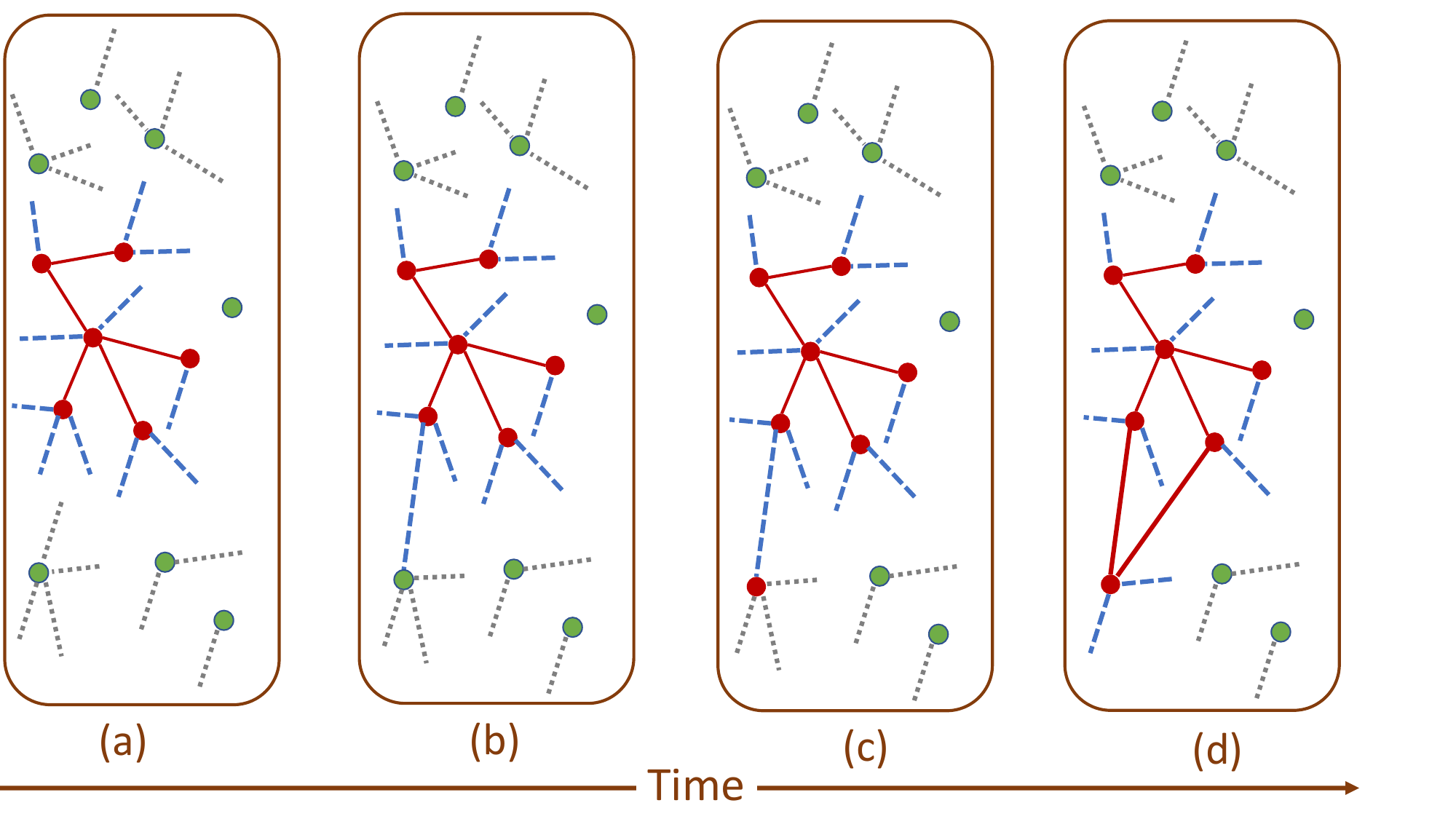}
	\caption{\label{fig:percolation}%
	\small Dynamics of the stochastic SI model over a finite time interval $\setTime{0}$. Susceptible vertices are shown in yellowish green and infected vertices, in dark red. Different types of edges  are  shown in different colours. \textbf{(a)} State of the system at a given time. We know the degree of each vertex. For each infected vertex, we know its number of infected neighbours (number of $\textrm{II}$-type edges, which are shown in solid red) and the number of $\textrm{IS}$-type half-edges, which are shown in dashed blue. The green dotted half-edges emanating from the susceptible vertices are of undetermined types. \textbf{(b)} An $\textrm{IS}$-type half-edge rings and is randomly matched with a half-edge originating from a susceptible vertex. A full edge is thus formed. \textbf{(c)} We change the status of the chosen susceptible vertex to infected (and change its colour to dark red). \textbf{(d)} We determine the types of the remaining half-edges of the newly infected vertex.
	}
\end{figure}


\subsection{Model}
\label{subsec:model}

We begin with the class of all  configuration model random graphs  \cite[Chapter 7]{hofstad2017randomVol1} with $n$ vertices, for  $n \in \setOfNaturals$. The main advantage of the configuration model is that it allows one to fix the degrees before constructing the graph itself. There are numerous real life situations where random graphs with a prescribed degree sequence (or a probability distribution for the degree sequence) are reasonable and intuitive. See \cite[Chapter 7]{hofstad2017randomVol1} for some examples.

Given degrees $d_1,d_2, \ldots, d_n$ for $n$ vertices, we first assign $d_i$ half-edges to vertex~$i$. The configuration model random graph is then obtained by uniformly-at-random matching (or pairing) of all  available half-edges. Two paired half-edges form an edge.  We assume the degrees $d_1, d_2, \ldots, d_n$ are drawn from a probability distribution, which we call the degree distribution. If the sum of the drawn degrees is not even, we add a parity half-edge to the last vertex. That is, we first draw $d_1, d_2, \ldots, d_{n-1}$ and $d_{n}^{'}$ as an independent and identically distributed (i.i.d.) random sample from the degree distribution and then set $d_n \defeq d_{n}^{'} + \indicator{d_{n}^{'} + \sum_{j=1}^{n-1} d_j \text{ is odd} }$. However, for large $n$  the contribution of the added indicator function  is negligible with probability one (see, e.g.,  \cite[Section 7.6, pp. 239]{hofstad2017randomVol1}),  and we will  thus ignore it  in  our asymptotic calculations.    

Additionally, we assume the first three moments of the degree distribution are finite. Note that, by matching half-edges uniformly at random,  we obtain  a \emph{multigraph} because the resulting graph may have self-loops and multiple edges. Interestingly, when the degrees $d_1, d_2, \ldots, d_n$ are drawn independently, \cite[Theorem 3.1.2]{durrett2007randomGraphDynamics} states that, as the number of vertices grows to infinity, the numbers of self-loops and parallel edges have independent Poisson limits whose means depend only on the first two moments of the degree distribution. Since the focus of this paper is on the functional limit of the fluctuations of the scaled stochastic process $n^{-1}X$, the contributions of self-loops and parallel edges are negligible in the limit. Therefore, we will ignore self-loops and parallels edges in our calculations of the asymptotic terms.  

Let us denote  the probability generating function (PGF) of the underlying degree distribution by $\psi$, \ie,
\begin{equation}
\psi(x) \defeq \sum_{k} x^k p_k \eqcomma
\label{eq:pgf}
\end{equation}
where $p_k$ is the probability that a randomly chosen vertex has degree~$k$, for $k \in \setOfNonnegativeIntegers$. We denote the class of all CM random graphs with $n$ vertices by $\mathcal{G}(\psi,n)$. 

A half-edge is referred to as $\mathrm{SI}$-type  (resp., $\mathrm{IS}$-type) if it  originates at a  susceptible (resp., infected) vertex and is a part of an $\mathrm{SI}$-type   edge.  Similarly, an $\mathrm{SS}$-type half-edge is a part of an $\mathrm{SS}$-type edge and an $\mathrm{II}$-type half-edge is a part of an $\mathrm{II}$-type edge (connecting two infected individuals).   


We consider the Markovian susceptible-infected model on CM random graphs. Each infected individual (represented by a vertex of the graph) infects one of its neighbours at rate~$\beta >0$, independently of the other neighbours. We split the population into two compartments  consisting respectively of the susceptible and the infected individuals.  In order to prove the  FCLT result outlined   in the introduction, we need to approximate the probability distributions of the counts $X_{\mathrm{S}}, X_{\mathrm{SI}}$ and $X_{\mathrm{SS}}$, which  requires  us to calculate the  probability that the susceptible vertex of degree $k$ has $l$ infected neighbours, for $l\le k$.  As we explain below, obtaining  this probability is particularly easy if we allow the graph to be constructed dynamically as the infection spreads. Consequently, we adopt the following dynamic construction of the graph first proposed in \cite{decreusefond2012large}.  

\paragraph{Dynamic construction of the graph}
  Suppose  we are given $n$ vertices with degrees (or number of half-edges) $d_1, d_2, \ldots, d_n$.   A  subset of size $X_{\mathrm{S}}(0)$ of $n$ vertices is chosen uniformly at random. Those vertices are  designated as   susceptible,  and   the remaining vertices,  infected.   Note that since we can write 
  $$X_{\mathrm{SI}}(t) = \sum_{i\in I(t)} X_{\mathrm{IS},i}(t) \eqcomma  $$ where  $X_{\mathrm{IS},i}(t)$ denotes the number of $\mathrm{IS}$-type half-edges originating at  an infected vertex $i$ at time $t$,  it suffices to follow the infected vertices in the dynamic construction. 

  We start at time $t=0$ by  infecting  a random subset of vertices   and revealing their  connections by matching their half-edges uniformly at random with the available half-edges (see below). 
  We then associate with each  $\mathrm{IS}$-type half-edge  an independent exponential clock with parameter $\beta$.   The first of all these clocks that rings determines the next event. Therefore,  at time $t$ we know that  the next infection  will take place in a time exponentially distributed with parameter $\beta X_{\mathrm{SI}}(t)$.  Let $t^\ast>t$ denote the time of the next infection  with its   left limit denoted by $t^\ast\!\!-$.  At $t^\ast$, we update the state of the epidemic as follows. 
  
  \begin{enumerate}
  \item[Step 1.]  We randomly match the $\mathrm{IS}$-type half-edge that has rung to a half-edge originating  from  a susceptible. This susceptible is the newly infected with degree, say, $k$.
  \item[Step 2.]  We choose uniformly $k - 1$ half-edges  among all the available half- edges (they either are of type $\mathrm{IS}$ or $\mathrm{SS}$). Let $n_{\mathrm{IS}}$ (resp., $n_{\mathrm{SS}}$) be the number of $\mathrm{IS}$-type (resp., of $\mathrm{SS}$-type) half-edges drawn among these $k - 1$ half-edges. Note that the probability of drawing a  specific pair $(n_{\mathrm{IS}},n_{\mathrm{SS}})$ is given by the hypergeometric distribution  $$\frac{\binom{X_{\mathrm{SI}}(t^\ast\!\!-)}{n_{\mathrm{IS}}}\binom{X_{\mathrm{SS}}(t^\ast\!\!-)}{n_{\mathrm{SS}}} }{\binom{X_{\mathrm{S}\Bigcdot}(t^\ast\!\!-)}{k-1}}.$$ 
The chosen $n_{\mathrm{IS}}$ half-edges of type $\mathrm{IS}$  determine the infected neighbors of the newly infected individual. The remaining $n_{\mathrm{SS}}$ edges of type $\mathrm{SS}$ remain open in the sense that the susceptible neighbor is not fixed.

\item[Step 3.]  We change the status of the $n_{\mathrm{IS}}$ (resp., $n_{\mathrm{SS}}$) $\mathrm{IS}$-type (resp., $\mathrm{SS}$-type) edges created to $\mathrm{II}$-type (resp., $\mathrm{SI}$-type).

\item [Step 4.] We change the status of the newly infected from $\mathrm{S}$ to $\mathrm{I}$ and wait for the next clock to ring at a new time $t^{\ast\ast}>t^\ast$ to repeat  Steps 1-4. The process stops when there are no more $\mathrm{IS}$-type  half-edges. 
\end{enumerate}

The above algorithm (see Figure~\ref{fig:percolation} for a pictorial representation) gives rise to the natural filtration $\history{t}$
as the  $\sigma$-field generated by the process history up to and including time~$t>0$ \cite[Chapter 1, p.~14]{durrett2010probability}.  The filtration $\history{t}$ includes information about the status of all vertices  (susceptible or infected), and all paired and unpaired half-edges.    
To be precise, we define 
\begin{align}
	\history{t} \defeq \sigma\left(\{  S(s), I(s), \{X_{\mathrm{IS},i}, X_{\mathrm{II},i}\}_{i \in I(s)}  \mid s \le t \}\right)  \eqstop 
	\nonumber 
\end{align}
We let $\history{0}$  contain \emph{all} $\prob$-null sets in $\history{}$. We include all $\prob$-null sets in $\history{0}$ so that
the filtration family $\{ \history{t} \}$ is complete. It  is also right continuous (\ie, $\history{t+}=\history{t}$ for every $t\geq 0$, where
$\history{t+} \defeq \cap_{s >0 } \history{ t+s}$
	is the \emph{$\sigma$-field of events immediately after~$t$}),   because it is generated by a right continuous jump process \cite[Chapter II, p.~61]{andersen1997statistical}. 
Therefore, the usual Dellacherie's conditions on 
$\{\history{t} \}$ are satisfied \cite{rebolledo1980central,shreve1998stochCalculus,fleming2005countingprocess}. We denote the $\sigma$-field of events strictly prior to $t \in \setOfPositiveReals$ by $\history{t-}$ and write $\history{0-} \defeq \history{0}$, by convention.

Denote the collection of vertices of degree $k$ that remain susceptible  at time $t$ by $S_k(t)$. Recall that $X(t) \defeq ( X_{\mathrm{S}}(t),  X_{\mathrm{SI}}(t), X_{\mathrm{SS}}(t) )$  denotes the aggregated state vector of the system at time $t \ge 0$. Note that,  since we may calculate  $X_{\mathrm{SS}}(t)$ if we know  $X_{\mathrm{SI}}(t)$ and $X_{\mathrm{II}}(t)$, the process $X$  is adapted to  $\history{t}$. However,  the process  $X$ itself is not necessarily Markovian.  Let $\tilde{X} \defeq  (X_{\mathrm{SI},i}, X_{\mathrm{SS},i})_{i \in S} $.
By the \emph{Doob-Meyer decomposition theorem} \cite{meyer1962decomposition}, we decompose  $X$ as
\begin{align}
X(t) = {} & X(0)+ \int_{0}^{t} \FOperator{\tilde{X}(s) }{ X } \differential{s} + M'(t) \eqcomma
\label{eq:doob-decomposition}
\end{align}
where 
$M'(t) \defeq ( M'_{\mathrm{S}}(t),  M'_{\mathrm{SI}}(t), M'_{\mathrm{SS}}(t)) $  is a zero-mean martingale adapted to the filtration $\history{t}$
%
and 
$\FOperator{\tilde{X}}{X} \defeq (\FOperator{\tilde{X} }{\mathrm{S}}, \FOperator{\tilde{X}}{\mathrm{SI}}, \FOperator{\tilde{X}}{\mathrm{SS}} )$ is an integrable function given by
\begin{align}
\begin{aligned}
\FOperator{\tilde{x}}{\mathrm{S}} & \defeq  -\beta x_{\mathrm{SI}}  \eqcomma \\
\FOperator{\tilde{x}  }{\mathrm{SI}} & \defeq   \sum_{i \in S}  \beta x_{\mathrm{SI},i} (x_{\mathrm{SS},i}- x_{\mathrm{SI},i})  \eqcomma \\
\FOperator{\tilde{x} }{\mathrm{SS}} & \defeq   -2 \sum_{i \in S} \beta x_{\mathrm{SI},i} x_{\mathrm{SS},i}  \eqcomma 
\end{aligned}
\end{align}
for $\tilde{x}\defeq  (x_{\mathrm{SI},i}, x_{\mathrm{SS},i})_{i \in S}$ and $x_{\mathrm{SI}} \defeq \sum_{i \in S} x_{\mathrm{SI},i}, x_{\mathrm{SS}} \defeq \sum_{i \in S} x_{\mathrm{SS},i}$. 
The first equation follows from the fact that the number of susceptible vertices decreases by $1$ at rate $\beta X_{\mathrm{SI}}$ because of the Markovian assumption of the SI process and the independence of the infection transmissions along the edges. Similarly, when a susceptible vertex $i$ gets infected, which happens at rate $\beta X_{\mathrm{SI},i}$ (again due to the independence of infection transmissions along each of its $X_{\mathrm{SI},i}$ infectious edges), the $\mathrm{SI}$ edges connected to that susceptible vertex turn into  $\mathrm{II}$-type edges whereas its $\mathrm{SS}$-type  edges turn into $\mathrm{SI}$-type edges. Therefore, the net change to the total count of SI-edges is $(X_{\mathrm{SS},i} -  X_{\mathrm{SI},i})$ whenever a susceptible vertex $i$ gets infected. Because of the independence assumption, we sum over all susceptible vertices to arrive at the second equation. See Figure~\ref{fig:edgeCountChange} for a visualization. The  third equation can be explained similarly by computing the rates and the corresponding net changes to the total count of SS-edges when a susceptible vertex gets infected.

We note that as a consequence of the uniformly-at-random matching of half edges in the dynamic construction,  similarly to Step~2,  we obtain also the following hypergeometric distribution (see also \cite{jacobsen2016large}) for the number of infected and susceptible neighbours of a susceptible vertex of degree~$k$ at time $t$.

\begin{myLemma}
	For $k \in \setOfNaturals_0$ and $i \in S_k(t)$, conditionally on the process history until time $t$, the vector $(X_{\mathrm{SI},i}(t), X_{\mathrm{SS},i}(t))$ follows the  hypergeometric distribution 
	\begin{align}\label{eq:susceptibleNbd}
\probOf{   X_{\mathrm{SI},i}(t)= n_{\mathrm{SI}}, X_{\mathrm{SS},i}(t)= n_{\mathrm{SS}}   \mid \history{t} } ={}&  \frac{  \binom{X_{\mathrm{SI}}(t)}{n_{\mathrm{SI}}}   \binom{ X_{\mathrm{SS}}(t) }{n_{\mathrm{SS}} } }{   \binom{X_{\mathrm{S} \Bigcdot }(t) }{k} } \eqcomma
	\end{align}
	supported on $n_{\mathrm{SI}}+ n_{\mathrm{SS}}= k$ where $n_{\mathrm{SI}},n_{\mathrm{SS}} \in \setOfNonnegativeIntegers$.

		\label{remark:susceptibleNbd}
\end{myLemma}

\begin{proof} Consider  $X_\mathrm{SI}(t)$,  the number of  half-edges of type $\mathrm{IS}$ emanating from the set of infected vertices. By construction,  they   all connect to the susceptible vertices uniformly at random.   Thus the  probability that $n_{\mathrm{SI}}$ of them connect to a given vertex  $i\in S_k$  when a total of $X_{\mathrm{S} \Bigcdot }(t)$ half-edges are available to choose  from is 
$$ \frac{\binom{k} {n_\mathrm{SI}} \binom{X_{\mathrm{S} \Bigcdot}(t)-k }{{X_\mathrm{SI}(t)}-n_\mathrm{SI}}}{\binom{X_{\mathrm{S} \Bigcdot}(t)}{X_\mathrm{SI}(t)}}$$ 
which works out to be  the same as  the hypergeometric probability \eqref{eq:susceptibleNbd}.  
\end{proof}
%
%

\begin{figure}
\centering
\includegraphics[width=0.5\columnwidth]{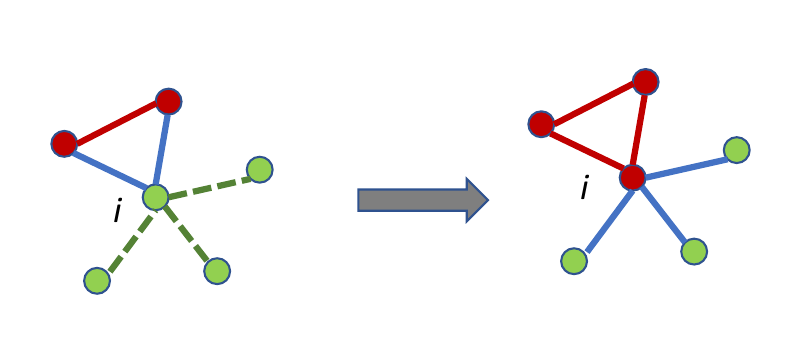}
  \caption{\label{fig:edgeCountChange}%
  \small Change to the edge counts. When a susceptible vertex $i$ gets infected (turns red from green; see Step 2 of the dynamic construction), the $\mathrm{SI}$ edges attached to that vertex turns into $\mathrm{II}$ edges whereas $\mathrm{SS}$ edges become new $\mathrm{SI}$ edges. This visualization is helpful in understanding the Doob--Meyer decomposition in Equation~\eqref{eq:doob-decomposition}.}
\end{figure}


Since we partition the collection of susceptible vertices $S(t)$ by their degree~$k \in \setOfNonnegativeIntegers$, we have $S(t) = \cup_k S_k(t)$. Therefore, $X_{\mathrm{S}}(t) = \sum_k X_{S_k}(t)$, where $X_{S_k}(t)$ is the size of $S_k(t)$. Note that $X_{\mathrm{S}\Bigcdot }(t) =  \sum_{k} k X_{S_k}(t)$.
To study the large graph limit of the system, we also define the following quantity
\begin{align}
\theta(t) \defeq {}& \myExp{ -\beta \int_{0}^{t} \frac{X_{\mathrm{SI}}(s)}{X_{\mathrm{S}\Bigcdot }(s)} \differential{s} } \eqcomma
\label{eq:theta_defn}
\end{align}
which can be intuitively described as the probability that a degree-$1$ vertex that was susceptible at time zero remains susceptible till time~$t >0$ (\cite{Volz2008,Miller2011note,Miller2012edge}). It may be described equivalently as
\begin{align}
\theta(t) ={} & \theta(0) + \int_{0}^{t} \FOperator{X_{\mathrm{SI}}(s),X_{\mathrm{S}\Bigcdot }(s), \theta(s) }{0} \differential{s} \eqcomma  \nonumber
\end{align}
where $\theta(0) =1$ and $\FOperator{X_{\mathrm{SI}},X_{\mathrm{S}\Bigcdot }, \theta}{0} \defeq - \beta \theta \dfrac{X_{\mathrm{SI}}}{X_{\mathrm{S}\Bigcdot }}$.

\section{The Law of Large Numbers}
\label{sec:LLN}

As done in \cite{jacobsen2016large}, we  make the following technical assumptions. Unless otherwise stated,  all limits below and elsewhere in the paper are taken in the large graph limit, \ie, as $n\to \infty$.  
\begin{enumerate}[label=\color{tud3d}\textbf{A\arabic*}]
	\item \label{itm:as1} There exists a positive constant $c_{\mathrm{S}\Bigcdot}$ such that $\probOf{ n^{-1} X_{\mathrm{S}\Bigcdot }(T) > c_{\mathrm{S}\Bigcdot} } \to 1$. 
	\item \label{itm:as2}  The fraction of initially susceptible 
	vertices is non-random and converges to some $\alpha_{\mathrm{S}}$, 
	\ie,
	\begin{align}
	n^{-1} X_{\mathrm{S}}(0) & \rightarrow  \alpha_{\mathrm{S}} \eqstop 
	\end{align}
	We also assume that the initially infected and susceptible vertices are selected uniformly at random and $0<\alpha_{\mathrm{S}} <1$. 
	\item \label{itm:as3} The degree distribution $\{p_k\}$ satisfies	$\sum_k k^3 p_k  < \infty$.
\end{enumerate}

%
We note that by virtue of the uniformly-at-random selection of the infected vertices at time zero,  \ref{itm:as2} also implies (see \cite{jacobsen2016large}) 
	\begin{align}
		\label{eq:alpha_defns}
	\begin{aligned}
	n^{-1} X_{\mathrm{I}}(0) & \rightarrow \alpha_{\mathrm{I}} = 1-  \alpha_{\mathrm{S}} >0  \eqcomma \\
	n^{-1} X_{\mathrm{SI}}(0) & \rightarrow  \alpha_{\mathrm{SI}}= \alpha_{\mathrm{S}} (1-  \alpha_{\mathrm{S}}) \partial \psi(1) \eqcomma \\
	n^{-1} X_{\mathrm{SS}}(0) & \rightarrow  \alpha_{\mathrm{SS}}= \alpha_{\mathrm{S}}^2 \partial \psi(1) \eqstop
	\end{aligned}
	\end{align}
	Below, for simplicity,  we use the vector notation $\alpha := (\alpha_{\mathrm{S}} , \alpha_{\mathrm{SI}},  \alpha_{\mathrm{SS}})$. The process $X_{\mathrm{I}} $ captures the number of infected individuals. Although it appears that  the assumption \ref{itm:as1} may be unnecessary,  we retain it for the sake of  consistency  with the work \cite{jacobsen2016large} since in our FCLT derivations  we rely on  some of the results obtained there. 


We note that  simulating the stochastic system satisfying \ref{itm:as1}, \ref{itm:as2}, and \ref{itm:as3} is straightforward with the help of the  dynamic construction described above. Indeed, for a given choice of the degree distribution satisfying \ref{itm:as3} with mean degree $\partial \psi(1) $, we may draw degrees $d_1, d_2, \ldots,  d_n$ and assign $d_i$ half-edges to vertex $i$. Then  we initialize the number of infected and susceptible nodes as in the dynamic graph construction above,  assigning  the initial proportions according  to $\alpha$. 
 We then  run the epidemic process following the dynamic construction up until a finite time $T$, which is independent of $n$, when there are still susceptible vertices left. This simulation algorithm  also constructs the stochastic process $X$ satisfying \ref{itm:as1}, \ref{itm:as2}, and \ref{itm:as3}.

For $f : \setOfReals \rightarrow \setOfReals$ and  $r \in \setOfNaturals $, the quantity $f^r$ is understood as $(r-1)$-times multiplication of $f$ with itself for $r \in \setOfNonnegativeIntegers$ with the convention $f^0 \defeq 1$. The symbol $\partial ^r f$  denotes the $r$-th derivative of the function~$f$, and by convention, we write $\partial f \defeq \partial ^1 f$.  Now, define 
the operator $\diffOperator{f}{r}$   as
\begin{align}
\diffOperator{f}{r} \defeq   f^{r-1}  \frac{ \partial ^r f}{ (\partial f  )^r } \eqcomma
\label{eq:diffOperatorDefn}
\end{align}
for $f : \setOfReals \rightarrow \setOfReals$
and  $r \in \setOfNaturals $ whenever the division of $\partial ^r f$ by $(\partial f  )^r$  is permissible, \ie, $\partial f   \neq 0$.  The operator $\diffOperator{}{r}$ is used to capture the impact of the graph structure on the limiting dynamics through the degree distribution.  Let 
$x \defeq  (x_{\mathrm{S}},  x_{\mathrm{SI}}, x_{\mathrm{SS}})$, and $\theta : \setOfPositiveReals \to [0,1]$ be a function. Then define $\kappa(\vartheta)$ as
\begin{align}
   \kappa(\vartheta) \defeq  \diffOperator{\psi \circ \vartheta }{2} \eqstop \label{eq:kappa_defn}
\end{align}
Note that 
\begin{align}
\kappa(\vartheta)(x) = {}& \frac{ \psi(\vartheta (x))   \partial^2 (\psi\circ \vartheta)(x) )  }{(\partial (\psi\circ\vartheta) (x))^2} \eqcomma x \in \setOfReals \eqstop  
\end{align}

If $(x,\vartheta)$ is the large graph FLLN limit of $(n^{-1} X, \theta)$, then,
following \cite[Section~3.3.3]{jacobsen2016large}, we interpret $\kappa(\vartheta)$ as the limiting ratio of the average excess degree of a susceptible vertex chosen randomly as a neighbour of an infectious individual, to the average degree of a susceptible vertex $\mu_{\mathrm{S}}$.  The quantity~$\kappa(\vartheta)$ allows us to count various pairs accurately.   In general, the operator $\diffOperator{\psi\circ \vartheta  }{r+1} $  recursively compares the excess degree of a susceptible vertex randomly chosen as a neighbour of $r$ infected individuals with that of a randomly chosen susceptible vertex. Therefore, it allows us to  count various $r$-configurations (different subgraph structures on $r$ vertices, \eg,   triples, quadruples etc.) accurately in the limit.   To be precise, in Lemma~\ref{thm:D_recurrence} in  Appendix~\ref{sec:DOperator},  we explicitly show
\begin{align}
	 \diffOperator{\psi\circ \theta  }{r+1} =  \frac{ \mu_{\mathrm{S}}^{(r)}   (\theta)  }{  \mu_{\mathrm{S}}  (\theta)  }  \diffOperator{\psi\circ \theta  }{r}  \ConvInProb  \diffOperator{\psi\circ \vartheta  }{r+1}    \eqcomma
	 \label{eq:D_recurrence}
\end{align}
where $ \mu_{\mathrm{S}}^{(r)}$ is the average excess degree of a  susceptible vertex randomly chosen as a neighbour of $r$ infected individuals. In Appendix~\ref{sec:DOperator},  we calculate these quantities explicitly.

Let us also   define  the operator $\HOperator{x,\vartheta}{} \defeq (\HOperator{x, \vartheta}{x }, \HOperator{x, \vartheta}{0} )$,  where $\HOperator{x,\vartheta}{x} \defeq
(\HOperator{x, \vartheta}{\mathrm{S}},\HOperator{x, \vartheta }{\mathrm{SI}},\HOperator{x, \vartheta}{\mathrm{SS}} )$, and $ \HOperator{x, \vartheta}{0} $ are given by
\begin{equation}
\begin{aligned}
\HOperator{x, \vartheta}{S} & \defeq  -\beta x_{\mathrm{SI}}   \eqcomma \\
\HOperator{x, \vartheta}{SI} & \defeq   \beta   \kappa(\vartheta) \frac{x_{\mathrm{SI}}}{x_S} (x_{\mathrm{SS}}- x_{\mathrm{SI}}) - \beta  x_{\mathrm{SI}}   \eqcomma \\
\HOperator{x, \vartheta }{SS} & \defeq  -2 \beta \kappa(\vartheta)  \frac{ x_{\mathrm{SI}} x_{\mathrm{SS}} }{x_S}  \eqcomma \\
\HOperator{x, \vartheta}{0} & \defeq  -\beta\frac{x_{\mathrm{SI}}}{\alpha_{\mathrm{S}} \partial \psi(\vartheta)}  \eqstop
\end{aligned}
\end{equation}

Now, noting that \ref{itm:as1} implies $P(X_{\mathrm{S}\Bigcdot} >k)\to 1 $ for any $k>0$ and \ref{itm:as3}  implies $\sum_k k^2 p_k < \infty$, recall the strong law on large graphs due to 
\cite{jacobsen2016large}. In the following we take the norm $\norm{.}$ to be uniform norm.

\begin{myTheorem}
	Assume \ref{itm:as1}, \ref{itm:as2}, and  \ref{itm:as3}
for a configuration model graph $\mathcal{G}(\psi, n)$. Then, for any $ T >0 $, the following holds
	\begin{equation*}
	\sup_{ 0 <t \le T } \norm{  ( X(t)/n, \theta(t))  -  (x(t), \vartheta(t))     } \ConvInProb 0 \eqcomma
	\end{equation*}
	where $(x, \vartheta) \defeq ( (x_{\mathrm{S}},  x_{\mathrm{SI}}, x_{\mathrm{SS}}), \vartheta )$ is the solution of
	\begin{equation}
		\label{eq:SIFCLT:LLN_ODE}
	(x (t) , \vartheta(t)  )= (x(0) , \vartheta(0))  + \int_{0}^{t} \HOperator{x(s) , \vartheta(s) }{}  \differential{s} \eqcomma
	\end{equation}
	with  the initial condition $x(0) = \alpha $ and $\vartheta(0) =1$.
	\label{thm:LLN}
\end{myTheorem}

	Observe that in the absence of recovery, the numbers of susceptible and infected individuals are linearly related as $X_{\mathrm{S}}+X_{\mathrm{I}} =n$ in the standard susceptible-infected-recovered  (SIR) model. The proof therefore follows immediately by setting the recovery rate in the SIR model to zero and assuming that there is only one layer in \cite{jacobsen2016large}.  The crucial observation is that the neighbourhood distribution of a susceptible vertex, conditional on the process history, can be expressed as a hypergeometric distribution  (see Lemma~\ref{remark:susceptibleNbd})  whose mixed moments can be approximated by the corresponding multinomial ones. This allows us to ``average out'' the individual-based quantities such as $X_{\mathrm{SI},i}$ for $i \in S$.  The convergence is then established by calculating several quadratic variations.   The proof of our FCLT presented in
	 Section~\ref{sec:FCLT}
	exploits similar calculations.


\section{Functional Central Limit Theorem}
\label{sec:FCLT}
Having obtained the functional law of large numbers, we now derive a functional central limit theorem for~$X$ after an appropriate scaling.  To this end, we begin by first defining the scaled martingale process 
\begin{align}
	\label{eq:SIFCLT:M_defn}
M(t) ={}& ( M_{\mathrm{S}}(t),  M_{\mathrm{SI}}(t), M_{\mathrm{SS}}(t))  \defeq  n^{-1/2} M' (t)  \eqcomma 
\end{align}
which is square integrable. 
We study the quadratic variation of the scaled martingale $M(t)$. The idea is to check whether either the optional or the predictable quadratic variation of the scaled process $ M$ 
converges in probability to a deterministic limit. If either of them does,  and if   the paths of $M$  become approximately continuous in the limit (``big'' jumps disappear),   we can make use of   
the \emph{Rebolledo theorem} (see Appendix~\ref{sec:rebolledo}; also \cite{rebolledo1980central,helland1982central})   to establish the asymptotic limit.



For each $\epsilon > 0$,
define~ 
\begin{align}
M^\epsilon(t) \defeq {}&  ( M_{\mathrm{S}}^\epsilon(t), M_{\mathrm{SI}}^\epsilon(t), M_{\mathrm{SS}}^\epsilon(t))
\label{eq:Mepsilon-defn}
\end{align}
to be a vector of square integrable martingales (only) containing all jumps of $M(t)$ larger in absolute value than~$\epsilon$. 
%

We use the shorthand notation $\predictableVariation{M}(t)$   for the $3 \times 3$ matrix of predictable covariation processes of the components of $M(t)$. That is,
\begin{align}
\predictableVariation{M}(t) \defeq{} & \begin{pmatrix}
\predictableVariation{M_{\mathrm{S}}}(t) & \predictableVariation{M_{\mathrm{S}} ,  M_{\mathrm{SI}} } (t)  & \predictableVariation{M_{\mathrm{S}} ,M_{\mathrm{SS}} }(t)  \\
\predictableVariation{M_{\mathrm{SI}} ,M_{\mathrm{S}}}(t) & \predictableVariation{ M_{\mathrm{SI}} }(t)  & \predictableVariation{M_{\mathrm{SI}} , M_{\mathrm{SS}} } (t) \\
\predictableVariation{M_{\mathrm{SS}}, M_{\mathrm{S}}}(t) &  \predictableVariation{ M_{\mathrm{SS}} , M_{\mathrm{SI}} }(t)  & \predictableVariation{M_{\mathrm{SS}} }(t)  \\
\end{pmatrix}\eqstop
\end{align}
Here $\predictableVariation{M_{\mathrm{S}}}(t) \defeq \predictableVariation{M_{\mathrm{S}}, M_{\mathrm{S}}}(t)$ etc., by convention.  See Appendix~\ref{sec:rebolledo} for definitions of quadratic variation processes.    Define $\predictableVariation{M^\epsilon}$ similarly. We shall study the large graph limits of $\predictableVariation{M}(t)$ and $\predictableVariation{M^\epsilon}(t)$ as $n \rightarrow~\infty$ for each $t \in \setTime{0}$. For this purpose,  we need the neighbourhood distribution of a susceptible vertex~$i$ of degree~$k$, \ie, the distribution of $(X_{\mathrm{SI},i}, X_{\mathrm{SS},i})$ for a vertex~$i \in S_k$, for all $k \in \setOfNaturals$. 

We quote an important remark from \cite{jacobsen2016large} that would come in handy for the derivations.

 \begin{myRemark}\label{remark:uniform-bounds}
Note that the total number of edges in the graph is $2^{-1} \sum_i d_i$. With $0< \alpha_{\mathrm{S}}<1$, it immediately follows that $n^{-1} X_{\mathrm{SI}} \leq n^{-1} X_{\mathrm{S} \Bigcdot  } \leq  \partial \psi(1)$ and $n^{-1} X_{\mathrm{SS}} \leq n^{-1} X_{\mathrm{S} \Bigcdot  } \leq  \partial \psi(1)$ for sufficiently large $n \in \setOfNaturals$. Also note that  $0\le \theta\le 1$  and  $\partial\psi(\theta) \le \partial\psi(1)$, so that    $\alpha_{\mathrm{S}} \theta \partial \psi(\theta) \leq \partial \psi(1)$. By virtue of \ref{itm:as1}, $n^{-1} X_{\mathrm{S} \Bigcdot}$ is bounded away from $0$ on $\setTime{0}$ and hence, so is $\theta$ (see \cite[Section~3.1]{jacobsen2016large}). As a consequence of  \cite[Lemma~1(b)]{jacobsen2016large}, we can take the same lower bound for $\alpha_{\mathrm{S}} \theta \partial \psi(\theta)$. Let us denote by $\xi>0$ the uniform lower bound for $n^{-1} X_{\mathrm{S} \Bigcdot}$ and $ \alpha_{\mathrm{S}} \theta \partial \psi(\theta)$ so that we can write $ n^{-1} X_{\mathrm{S} \Bigcdot}  \in [  \xi, \partial \psi(1) ] \subset \setOfPositiveReals $.  Note that all these bounds hold with probability $1$. 
 \end{myRemark}

%


  \subsection{Deterministic Limit of $\predictableVariation{M}(t)$}
  Recall that 
  $(x, \vartheta) \defeq ( (x_{\mathrm{S}}, x_{\mathrm{SI}}, x_{\mathrm{SS}}), \vartheta )$. Let us begin by defining the following  operators acting on the function $(x, \vartheta) $, 
  \begin{align}
  \begin{aligned}
  v_{\mathrm{S}  }  \defeq {} &  \beta x_{\mathrm{SI}}  \eqcomma  \\
  v_{\mathrm{SI}  }\defeq {} &   \beta \left(  \frac{x_{\mathrm{SI}} (x_{\mathrm{SS}}-x_{\mathrm{SI}} )^2   }{x_{\mathrm{S}}^2}  \diffOperator{  \psi(\vartheta )  }{3}
    - \frac{    x_{\mathrm{SI}} (x_{\mathrm{SS}}-3x_{\mathrm{SI}} )    }{x_{\mathrm{S}}}    \diffOperator{  \psi(\vartheta)  }{2}     + x_{\mathrm{SI}}  \right)   \eqcomma   \\
  v_{\mathrm{SS}  }   \defeq {} &   4 \beta \frac{ x_{\mathrm{SI}} x_{\mathrm{SS}}  }{x_{\mathrm{S}} } \left(  \frac{ x_{\mathrm{SS}}   }{x_{\mathrm{S}}}  \diffOperator{ \psi(\vartheta )  }{3}   + \diffOperator{ \psi(\vartheta)  }{2}    \right)  \eqcomma  \\
  v_{\mathrm{S} , \mathrm{SI} } \defeq {} &  - \beta  \left(\frac{  x_{\mathrm{SI}} ( x_{\mathrm{SS}} - x_{\mathrm{SI}})  }{  x_{\mathrm{S}} } \diffOperator{\psi(\vartheta )   }{2}- x_{\mathrm{SI}}  \right)   \eqcomma   \\
  v_{\mathrm{S} , \mathrm{SS} }  \defeq {} &    2\beta \frac{  x_{\mathrm{SI}} x_{\mathrm{SS}}  }{x_{\mathrm{S}} } \diffOperator{    \psi(\vartheta ) }{2}   \eqcomma    \\
  v_{\mathrm{SI}, \mathrm{SS}  }  \defeq {} & -2 \beta \frac{  x_{\mathrm{SI}}x_{\mathrm{SS}} (x_{\mathrm{SS}}- x_{\mathrm{SI}})  }{ x_{\mathrm{S}}^2  }    \diffOperator{   \psi(\vartheta ) }{3 } \eqstop     \\
  \end{aligned}
  \label{eq:vdefinition}
  \end{align}

  The intuition behind the operators in Equation~\eqref{eq:vdefinition} will be clear when we compute the predictable quadratic variation $\predictableVariation{M} $ of the martingale process $M$ and seek its limit as $n \rightarrow \infty$. In particular, we shall see that each of the terms on the right-hand side of Equation~\eqref{eq:vdefinition} is a function of various multinomial moments, which we find as a limit of the corresponding hypergeometric ones in order to approximate the expected jump sizes conditional on the process history.

  Now,  define a $\setTime{}$-indexed family of matrices $\{ V(t) \}$ as follows
    \begin{align}
    V(t) \defeq{} & \begin{pmatrix}
    V_{\mathrm{S}}(t) &  V_{\mathrm{S} ,  \mathrm{SI} } (t)  & V_{\mathrm{S} ,\mathrm{SS} }(t)  \\
    V_{\mathrm{SI} ,\mathrm{S}}(t) & V_{ \mathrm{SI} }(t)  & V_{\mathrm{SI} , \mathrm{SS} } (t) \\
    V_{\mathrm{SS}, \mathrm{S}}(t) & V_{ \mathrm{SS} , \mathrm{SI} }(t)  & V_{\mathrm{SS} }(t)  \\
    \end{pmatrix} \eqcomma \label{eq:Vmatrix}
    \end{align}
  where,  given $  v_{id_1,id_2}( x,\vartheta  )  \text{ for }  id_1,id_2 \in \{\mathrm{S}, \mathrm{SI}, \mathrm{SS} \} $  
  in Equation~\eqref{eq:vdefinition},
  \begin{align}
  V_{id_1, id_2} (t) \defeq  {}& \int_{0}^{t}  v_{id_1,id_2}( x(s),\vartheta(s)  )  \differential{s} \eqcomma  \label{eq:V_ODE} 
  \end{align}
with the convention $   v_{id_1,id_2} \defeq  v_{id_2,id_1} \text{ for }  id_1,id_2 \in \{\mathrm{S}, \mathrm{SI}, \mathrm{SS} \} $ and $ v_{id_1,id_2} \defeq  v_{id_1} $ whenever
  $ id_1= id_2 \in \{\mathrm{S}, \mathrm{SI}, \mathrm{SS} \}   $. Note that this also sets the convention $    V_{id_1,id_2}(t ) \defeq  V_{id_2,id_1}( t) \text{ for }  id_1,id_2 \in \{\mathrm{S}, \mathrm{SI}, \mathrm{SS} \} $ and $ V_{id_1,id_2}( t  ) \defeq  V_{id_1}( t )  $ whenever
  $ id_1= id_2 \in \{\mathrm{S}, \mathrm{SI}, \mathrm{SS} \}   $  for each $t \in \setTime{}$.

  Let us now present our first result providing the deterministic limit of $\predictableVariation{M} $ in the following lemma.  Recall that proving a deterministic limit of  $\predictableVariation{M} $  would satisfy one of the conditions of the Rebolledo  theorem.   The key strategy in establishing  the result will be to approximate various hypergeometric moments by the corresponding multinomial ones. 

  \begin{myLemma}
  Consider the stochastic SI model described in Section~\ref{subsec:model}. Assume \ref{itm:as1}, \ref{itm:as2} and \ref{itm:as3} for a configuration model graph $\mathcal{G}(\psi, n)$. Then,
  \begin{align}
  \predictableVariation{M}(t)  \ConvInProb {}& V(t)  \eqcomma \nonumber
  \end{align}
for each $t \in \setTime{}$,  as $ n \rightarrow \infty$ where $V(t)$ is as defined in Equation~\eqref{eq:Vmatrix}, and $ (x, \vartheta)  $ is the solution of Equation~\eqref{eq:SIFCLT:LLN_ODE}
  	with 
		$x (0) = \alpha $ and $\vartheta (0)=1$.
  	\label{lemma:deterministic-limit}
  \end{myLemma}


  \begin{proof}[Proof of Lemma~\ref{lemma:deterministic-limit}]
  	To show  convergence of the matrix random process $  \predictableVariation{M}(t)   $ to $V(t)$, we show component-wise convergence of the respective components. The general strategy to prove convergence for these components remains the same.
To conserve space, we only  demonstrate here the strategy for  establishing $M_{\mathrm{SI}}(t)   \ConvInProb V_{\mathrm{SI}}(t) $. 
Remaining   assertions follow similarly. 

  \noindent\paragraph{Computation of $\predictableVariation{M_{\mathrm{SI}}}$}
  The process $M_{\mathrm{SI}}$  jumps only if a susceptible vertex gets infected. 
  Therefore, the predictable quadratic variation 
	is computed as follows
\begin{align*}
	 \predictableVariation{M_{\mathrm{SI}}}(t) = \predictableVariation{   n^{-1/2} M'_{\mathrm{SI}}   }(t) ={} & \int_{0}^{t}  \sum_k   \frac{1}{n}  \sum_{i \in S_k}  \beta X_{\mathrm{SI},i} (X_{\mathrm{SS},i}- X_{\mathrm{SI},i} )^2    \differential{s}   \eqstop
\end{align*}

   Now, for a randomly selected~$i \in S_k$, we seek to find the (conditional) moments $\Eof{  X_{\mathrm{SI},i} (X_{\mathrm{SS},i}- X_{\mathrm{SI},i} )^2  \mid \history{t-}  }$. 	 Define the function $C_{h}^k :\setTime{0} \rightarrow \setOfReals   $ as 
   \begin{align*}
	C_{h}^k(t) \defeq &   \Eof{  X_{\mathrm{SI},i}(t) (X_{\mathrm{SS},i}(t) - X_{\mathrm{SI},i}(t) )^2 \mid \history{t-}   }   \eqstop 
   \end{align*}
   Following the computations in Appendix~\ref{sec:hypergeometric-moments}, we get
	 \begin{align*}
	C_{h}^k(t)
	 ={} &  \frac{   \NPermuteR{k}{3} X_{\mathrm{SI}}  }{\NPermuteR{X_{\mathrm{S} \Bigcdot }}{3} }    \left[     \NPermuteR{X_{\mathrm{SS}}}{2}  - 2 (X_{\mathrm{SI}}-1    )X_{\mathrm{SS}}
		+    \NPermuteR{ X_{\mathrm{SI}}-1    }{2}   \right] \\
	&  -   \frac{   \NPermuteR{k}{2} X_{\mathrm{SI}} }{\NPermuteR{X_{\mathrm{S} \Bigcdot }}{2} } [ X_{\mathrm{SS}} - 3(X_{\mathrm{SI}}-1)     ]  + k \frac{X_{\mathrm{SI}}}{  X_{\mathrm{S} \Bigcdot }  } \eqstop
	 \end{align*}
To be precise, the processes on the right-hand side of the above equation are evaluated at $t-$. To approximate the hypergeometric moments by corresponding multinomial ones, define the  multinomial compensator $C_{m}^k : \setTime{0} \times [\xi,  \partial \psi(1)] \rightarrow \setOfReals   $ as
\begin{align*}
C_{m}^k (t,z) \defeq {}& \frac{   \NPermuteR{k}{3}   n^{-3} X_{\mathrm{SI}}  }{z^{3} }    (    {X_{\mathrm{SS}}}^{2}  - 2 X_{\mathrm{SI}}X_{\mathrm{SS}}   + X_{\mathrm{SI}}^2    ) \\
& {}-   \frac{   \NPermuteR{k}{2} n^{-2} X_{\mathrm{SI}} }{z^{2} } ( X_{\mathrm{SS}} - 3X_{\mathrm{SI}}     )
+  \frac{k n^{-1}   X_{\mathrm{SI}}}{  z}     \\
={} &     \frac{   \NPermuteR{k}{3}   n^{-3} X_{\mathrm{SI}}   (  X_{\mathrm{SS}}- X_{\mathrm{SI}}   ) ^2  }{z^{3} }      -   \frac{   \NPermuteR{k}{2} n^{-2} X_{\mathrm{SI}}  ( X_{\mathrm{SS}} - 3X_{\mathrm{SI}}     ) }{z^{2} }
 +   \frac{  k n^{-1}  X_{\mathrm{SI}}}{  z}   \eqstop
\end{align*}
Again, to be precise, the processes on the right-hand side of the above equation are evaluated at $t-$. Please observe that there exists an $L >0$ such that
\begin{align}
C_{m}^k (t,z(t) ) \leq L k^3 \eqcomma
\label{eq:C_mBound}
\end{align}
  uniformly in $n$. This holds because $n^{-1}X_{\mathrm{SI}}$ and $n^{-1}X_{\mathrm{SS}}$ are uniformly bounded above by virtue of Remark~\ref{remark:uniform-bounds} and $z$ is bounded away from zero, by definition. The function $C_{m}^k (t,z(t) )   $ is also Lipschitz continuous in $z$.  Now recall the definition of $v_{\mathrm{SI}}$ from Equation~\eqref{eq:vdefinition} and define
 \begin{align*}
 \Delta_{}(t) \defeq {}& \sum_k   \frac{1   }{n}  \sum_{i \in S_k} \beta   X_{\mathrm{SI},i}(t) (X_{\mathrm{SS},i}(t) - X_{\mathrm{SI},i} (t) )^2    - v_{\mathrm{SI}}(x(t), \vartheta(t) )   \\
 ={}& \sum_k   \frac{1   }{n}  \sum_{i \in S_k} \beta   X_{\mathrm{SI},i}(t) (X_{\mathrm{SS},i}(t)- X_{\mathrm{SI},i}(t) )^2   -v_{\mathrm{SI}}(  n^{-1}X(t), \theta(t)  )  \\
 &  +v_{\mathrm{SI}}(  n^{-1}X(t) , \theta (t)  )  - v_{\mathrm{SI}}(x(t), \vartheta(t) )  \\
 ={} &  \Delta_{1}(t) +  \Delta_{2}(t)  \eqcomma
 \end{align*}
 where $  \Delta_{1}(t) \defeq \sum_k   \frac{1   }{n}  \sum_{i \in S_k} \beta   X_{\mathrm{SI},i}(t) (X_{\mathrm{SS},i}(t) - X_{\mathrm{SI},i}(t) )^2   -v_{\mathrm{SI}}(  n^{-1}X(t), \theta(t)  ) $, and $  \Delta_{2}(t) \defeq  v_{\mathrm{SI}}(  n^{-1}X(t) , \theta(t)  )  - v_{\mathrm{SI}}(x(t), \vartheta(t)  ) $.
To show $\predictableVariation{M_{\mathrm{SI}}} \ConvInProb V_{\mathrm{SI}}   $, it suffices to show $ \sup_{t \in \setTime{}}    |    \Delta_{}(t) | \ConvInProb 0$. We  achieve this by separately showing $ \sup_{t \in \setTime{}}    |    \Delta_{1}(t) | \ConvInProb 0   $ and $\sup_{t \in \setTime{}}    |    \Delta_{2}(t) | \ConvInProb 0$.

\paragraph*{Convergence of  $\Delta_{1}(t)$} See that
 \begin{align*}
 \Delta_{1}(t)  =  {} & 
  \sum_k   \frac{1   }{n}  \sum_{i \in S_k} \beta   X_{\mathrm{SI},i} (X_{\mathrm{SS},i}- X_{\mathrm{SI},i} )^2
  - \beta [  \frac{  n^{-3}     X_{\mathrm{SI}} (X_{\mathrm{SS}}-X_{\mathrm{SI}} )^2   }{\alpha_{\mathrm{S}}^2}   \frac{\partial^3  \psi(\theta ) }{   (   \partial  \psi(\theta ))^3 }   \\
 & {} - \frac{  n^{-2}  X_{\mathrm{SI}} (X_{\mathrm{SS}}-3X_{\mathrm{SI}} )    }{\alpha_{\mathrm{S}} }   \frac{\partial^2  \psi(\theta ) }{   (  \partial  \psi(\theta ))^2 }
    +  n^{-1}  X_{\mathrm{SI}}  ]        \\
 ={} &  \sum_k  [ \frac{1   }{n}  \sum_{i \in S_k} \beta   X_{\mathrm{SI},i} (X_{\mathrm{SS},i}- X_{\mathrm{SI},i} )^2
  - \beta \{  \frac{  n^{-3}     X_{\mathrm{SI}} (X_{\mathrm{SS}}-X_{\mathrm{SI}} )^2   }{\alpha_{\mathrm{S}}^2}  \frac{  \NPermuteR{k}{3} \theta^k p_k  }{   (\theta \partial \psi(\theta)    )^3 } \\
 & {} -
 \frac{  n^{-2}  X_{\mathrm{SI}} (X_{\mathrm{SS}}-3X_{\mathrm{SI}} )    }{\alpha_{\mathrm{S}}}  \frac{  \NPermuteR{k}{2} \theta^k p_k  }{   (\theta \partial \psi(\theta)    )^2 }
  + n^{-1}  X_{\mathrm{SI}}  \frac{k \theta p_k}{   \theta \partial \psi(\theta)    }  \}   ]  \\
 ={} &  \sum_k  [ \frac{1   }{n}  \sum_{i \in S_k} \beta   X_{\mathrm{SI},i} (X_{\mathrm{SS},i}- X_{\mathrm{SI},i} )^2  - \beta \alpha_{\mathrm{S}} p_k \theta^k C_{m}^k (t,  \alpha_{\mathrm{S}} \theta \partial \psi(\theta)  ) ]   \eqstop
 \end{align*}

The second equality follows by expressing the derivatives of the PGFs as sum over all possible degrees $k$ and then collecting terms involving degree $k$.  Define $ \Delta_{1}^{(k)}(t)  \defeq \frac{1   }{n}  \sum_{i \in S_k} \beta   X_{\mathrm{SI},i} (X_{\mathrm{SS},i}- X_{\mathrm{SI},i} )^2  - \beta \alpha_{\mathrm{S}} p_k \theta^k C_{m}^k (t,  \alpha_{\mathrm{S}} \theta \partial \psi(\theta)  ) $.
Our task  boils down to showing that $ \sup_{t \in \setTime{}}  |  \sum_k  \Delta_{1}^{(k)}(t)   |  \ConvInProb 0 $ as $n \rightarrow \infty$. We achieve this  in two steps. First we show that the tails of $\sum_k  \Delta_{1}^{(k)}(t)  $ are negligible. Second, we show that each term $  \Delta_{1}^{(k)}(t)$ converges to zero uniformly in probability for a fixed $k \in \setOfNaturals$.

  \noindent\paragraph{(Step I) Tails are negligible}
  Let us begin by showing that as $N \rightarrow \infty$,
  \begin{align*}
  \sup_{n \in \setOfNaturals}   \sup_{t \in \setTime{}}  |  \sum_{k > N}  \Delta_{1}^{(k)}(t)   |  \ConvInProb 0 \eqstop
  \end{align*}
  Observe that, for sufficiently large $n$, 
	\begin{align}
	|    \frac{1   }{n}  \sum_{k > N}   \sum_{i \in S_k} \beta   X_{\mathrm{SI},i} (X_{\mathrm{SS},i}- X_{\mathrm{SI},i} )^2  |
	\leq {}    \frac{ \beta   }{n}  \sum_{k > N}    k^3 X_{S_k}
	\leq {}  2 \beta  \sum_{k > N}    k^3 p_k  \eqcomma
	 \label{eq:tail-part1}
	\end{align}
  because $  n^{-1} X_{S_k} \leq 2p_k $ for sufficiently large $n$ in the light  of Remark~\ref{remark:uniform-bounds}. Following Remark~\ref{remark:uniform-bounds} and the bound on $C_m^k$ from Equation~\eqref{eq:C_mBound}, we get 
  \begin{align}
  |    \sum_{k > N}   \beta \alpha_{\mathrm{S}} p_k \theta^k C_{m}^k (t,  \alpha_{\mathrm{S}} \theta \partial \psi(\theta)  )   | \leq {}& \beta L \sum_{k> N} k^3 p_k   \eqstop
  \label{eq:tail-part2}
  \end{align}
	Therefore, we get  $ \sup_{n \in \setOfNaturals}   \sup_{t \in \setTime{}}  |  \sum_{k > N}  \Delta_{1}^{(k)}(t)   |  \ConvInProb 0 $, combining inequalities~\eqref{eq:tail-part1} and \eqref{eq:tail-part2} in view  of \ref{itm:as3}.

   \noindent\paragraph{(Step II) Uniform convergence in probability for a fixed k}
   In addition to Step I, it is sufficient to show $\sup_{t \in \setTime{}} |   \Delta_{1}^{(k)}(t)   | \ConvInProb 0 $ for an arbitrarily fixed $k \in \setOfNaturals$ to justify $ \sup_{t \in \setTime{}} |   \Delta_{1}^{}(t)   | \ConvInProb 0       $. Observe that
   \begin{align}
   | \Delta_{1}^{(k)}(t) | ={} & | \frac{1   }{n}  \sum_{i \in S_k} \beta   X_{\mathrm{SI},i} (X_{\mathrm{SS},i}- X_{\mathrm{SI},i} )^2  - \beta \alpha_{\mathrm{S}} p_k \theta^k C_{m}^k (t,  \alpha_{\mathrm{S}} \theta \partial \psi(\theta)  )    |   \nonumber \\
   \leq {} & \beta n^{-1} |  \sum_{i \in S_k}   X_{\mathrm{SI},i} (X_{\mathrm{SS},i}- X_{\mathrm{SI},i} )^2    - X_{S_k} C_h^k(t)   |   \label{eq:delta-part1} \\
   & + \beta n^{-1} X_{S_k}  | C_h^k(t)   - C_{m}^k (t, n^{-1} X_{\mathrm{S} \Bigcdot}   )  |     \label{eq:delta-part2} \\
   & + \beta  |   n^{-1} X_{S_k}  C_{m}^k (t, n^{-1} X_{\mathrm{S} \Bigcdot}   )  - \alpha_{\mathrm{S}} p_k \theta^k  C_{m}^k (t, n^{-1} X_{\mathrm{S} \Bigcdot}   )      |   \label{eq:delta-part3} \\
   & + \beta  \alpha_{\mathrm{S}} p_k \theta^k |  C_{m}^k (t, n^{-1} X_{\mathrm{S} \Bigcdot}   )  - C_{m}^k (t, \alpha_{\mathrm{S}} \theta \partial \psi(\theta)   )    |   \label{eq:delta-part4}   \eqstop
   \end{align}
 We show that each of the above summands  converges uniformly in probability to zero.

  Define the process $ \Delta_{1,1}^{(k)}(t) \defeq  \sum_{i \in S_k}   X_{\mathrm{SI},i} (X_{\mathrm{SS},i}- X_{\mathrm{SI},i} )^2    - X_{S_k} C_h^k(t)  $. Observe that $ \Delta_{1,1}^{(k)}(t)$ is a zero-mean, piecewise constant, \cadlag\ martingale with paths in $D$. The jumps of $ \Delta_{1,1}^{(k)}(t)  $ take place when a vertex of degree-$k$ gets infected. The quadratic variation of $ \Delta_{1,1}^{(k)}(t)  $ is therefore the sum of its squared jumps 
  \begin{align*}
  \optionalVariation{   \Delta_{1,1}^{(k)}} (t) ={} & \sum_{s \leq t}   (\delta \Delta_{1,1}^{(k)}(s)  )^2 \leq k^6 n \eqcomma
  \end{align*}
  because the number of jumps can not exceed $n$. Therefore, by Doob's martingale inequality we get
	$\sup_{ t \in \setTime{}} | n^{-1}   \Delta_{1,1}^{(k)}(t)   |  \ConvInProb 0 $,
  since $\Eof{ \optionalVariation{   \Delta_{1,1}^{(k)}} (t)  } = \Eof{ (\Delta_{1,1}^{(k)}(t) )^2   } = \BigO{n}$.   That is, the quantity in \eqref{eq:delta-part1} converges uniformly in probability to zero.

  For the term in \eqref{eq:delta-part2}, take into account  $ n^{-1}X_{S_k}   \leq 1 $  and  see that
\begin{align*}
	 \sup_{ t \in \setTime{}} |  C_h^k(t)   - C_{m}^k (t, n^{-1} X_{\mathrm{S} \Bigcdot}   )   |   \leq {}&  \frac{ c_1 k^3   }{  X_{S \Bigcdot }(T) -2   }   \eqcomma
\end{align*}
	%
	%
  for some $c_1>0$, because $  X_{\mathrm{S} \Bigcdot }  $ is non-increasing on $\setTime{0} $. Therefore, by \ref{itm:as1}, the quantity in \eqref{eq:delta-part2} converges to zero uniformly in probability.

  Now observe that
  \begin{align*}
    \sup_{ t \in \setTime{}} |   n^{-1} X_{S_k}  C_{m}^k (t, n^{-1} X_{\mathrm{S} \Bigcdot}   )  - \alpha_{\mathrm{S}} p_k \theta^k  C_{m}^k (t, n^{-1} X_{\mathrm{S} \Bigcdot}   )      | \\
   \leq L k^3    \sup_{ t \in \setTime{}} | n^{-1} X_{S_k}  -\alpha_{\mathrm{S}} p_k \theta^k   | &
    \ConvInProb 0  \eqcomma
  \end{align*}
 by virtue of the bound on $C_m^k$ in Equation~\eqref{eq:C_mBound} and \cite[Lemma~1(a)]{jacobsen2016large}. Therefore, the term in \eqref{eq:delta-part3} also converges to zero uniformly in probability.

 Finally, by virtue of Lipschitz continuity of $C_m^k(t,z)$ in $z$, we get
 \begin{align*}
  \sup_{ t \in \setTime{}} |  C_{m}^k (t, n^{-1} X_{\mathrm{S} \Bigcdot}   )  - C_{m}^k (t, \alpha_{\mathrm{S}} \theta \partial \psi(\theta)   )    |   \leq {} & c_2  \sup_{ t \in \setTime{}} |  n^{-1} X_{\mathrm{S} \Bigcdot}  -  \alpha_{\mathrm{S}} \theta \partial \psi(\theta)   |
   \eqcomma
 \end{align*}
 for some $c_2 >0$. Because $  \sup_{ t \in \setTime{}} |  n^{-1} X_{\mathrm{S} \Bigcdot}  -  \alpha_{\mathrm{S}} \theta \partial \psi(\theta)   |   \ConvInProb 0    $ as shown in \cite{jacobsen2016large}, we conclude that the term in \eqref{eq:delta-part4} converges to zero uniformly in probability.

 Having shown the terms in \eqref{eq:delta-part1}, \eqref{eq:delta-part2}, \eqref{eq:delta-part3} and \eqref{eq:delta-part4} converge to zero uniformly in probability, we establish that $\sup_{ t \in \setTime{}}  | \Delta_{1}^{(k)}(t) |     \ConvInProb 0 $
 uniformly in probability for any fixed $k \in \setOfNaturals$. Finally, by virtue of Step I and Step II, we obtain  $ \sup_{ t \in \setTime{}}  | \Delta_{1}(t) |     \ConvInProb 0 $.

 \paragraph{Convergence of  $   \Delta_{2}(t)$}  Note that $ v_{\mathrm{SI}}(  n^{-1}X, \theta  ) $ is Lipschitz continuous on its domain that we can take as $ (0,1] \times [\xi,  \partial \psi(1)]^2 \times [\xi,1]    $, by  Remark~\ref{remark:uniform-bounds}.  Therefore,
 \begin{align*}
  \sup_{ t \in \setTime{}} |  v_{\mathrm{SI}}(  n^{-1}X , \theta  )  - v_{\mathrm{SI}}(x, \vartheta )   |
   \leq {}  c_3  \sup_{ t \in \setTime{}}  \norm{  (n^{-1}X, \theta )-  (x, \vartheta )   }   \eqcomma
 \end{align*}
 for some Lipschitz constant $c_3>0$. Since $ (x, \vartheta)  $ is the solution of Equation~\eqref{eq:SIFCLT:LLN_ODE},
 	with initial condition $x (0) = \alpha $ and $\vartheta (0)=1$, we get by virtue of Theorem~\ref{thm:LLN}, $ \sup_{ t \in \setTime{}} |   \Delta_{2}(t)  |  \ConvInProb 0$.

 	\paragraph{Final Conclusion}
 Since  $ \sup_{ t \in \setTime{}}  | \Delta_{1}(t) |     \ConvInProb 0$ and $ \sup_{ t \in \setTime{}} |   \Delta_{2}(t)  |  \ConvInProb 0$, 
	we conclude $  \sup_{ t \in \setTime{}}  |   \Delta(t)  |  \ConvInProb 0$, which is a sufficient condition for
\begin{align*}
	 \predictableVariation{M_{\mathrm{SI}}} (t)   \ConvInProb V_{\mathrm{SI}}(t)=  {}& \int_{0}^{t}  v_{\mathrm{SI}}(x(s), \vartheta(s)  )  \differential{s}  \eqstop
\end{align*}
  \end{proof}

We remark that the various moment estimates used in the proof above (and elsewhere in the paper) ignore the contributions of self-loops and parallel edges. Their contributions are asymptotically negligible as discussed earlier in Section~\ref{subsec:model}. Also, note that we may need to add a parity edge to the last vertex if the sum of the drawn degrees is not even. This is relevant for the calculations in Steps I and II. However, as discussed earlier in Section~\ref{subsec:model},  the contribution due to this minor adjustment is negligible and hence, is not shown explicitly. 


\subsection{Asymptotic Rarefaction of Jumps}
Recall that $M^\epsilon \defeq   ( M_{\mathrm{S}}^\epsilon,  M_{\mathrm{SI}}^\epsilon, M_{\mathrm{SS}}^\epsilon)$
is the vector of square integrable martingales containing all jumps of components of $M$ larger than~$\epsilon$ in absolute value, for~$\epsilon>0$, \ie, $ M_{id}(t) - M_{id}^\epsilon(t)   $ is a local square integrable martingale and $| \delta M_{id}(t) - \delta M_{id}^\epsilon(t) | \leq \epsilon $ for all ${id} \in \{ \mathrm{S},  \mathrm{SI}, \mathrm{SS}  \}$ and $ t \in \setTime{} $.  We wish to show
$ \predictableVariation{ M_{id}^\epsilon   } (t) \ConvInProb 0$ for all ${id} \in \{ \mathrm{S},  \mathrm{SI}, \mathrm{SS}  \}$ and $t \in \setTime{} $, as $n \rightarrow \infty$. We would like to point out that this condition is essentially the \emph{strong Asymptotic Rarefaction of Jumps Condition of the second type} (strong ARJ$(2)$) as described in \cite{rebolledo1980central,andersen1997statistical}.  Intuitively this ensures that the sample paths of the martingale $M(t)$ are close to continuous in the limit.
Before proceeding further, we offer  the following remark.

\begin{myLemma}
For the configuration model  graph $\mathcal{G}(\psi, n)$ along with~\ref{itm:as3}, the following holds true:
\begin{align}
n^{- \frac{1}{2}} d_{\mathrm{max}}  \ConvAlmostSure 0 \eqcomma
\end{align}
where $d_{\mathrm{max}}$ is the maximum degree observed in a realization of $\mathcal{G}(\psi, n)$.
\label{remark:max-degree}
\end{myLemma}
%

\begin{proof}[Proof of Lemma~\ref{remark:max-degree}]
	The result follows by a direct application of the result in \cite[Theorem~5.2]{barndorff-nielsen1963} along with ~\ref{itm:as3}.

%
\end{proof}

Let us now compute the predictable quadratic variation of $M^\epsilon$ and establish its asymptotic limit.

\begin{myLemma}
  Consider the stochastic SI model described in Section~\ref{subsec:model}. Assume \ref{itm:as1}, \ref{itm:as2} and \ref{itm:as3} for a configuration model graph $\mathcal{G}(\psi, n)$. Consider the vector  $M^\epsilon$ of square integrable martingales containing all jumps of components of $M(t)$ larger than $ \epsilon $ in absolute value for~$\epsilon>0$, as defined in Equation~\eqref{eq:Mepsilon-defn}.  Then, as $ n \rightarrow \infty$, for all ${id} \in \{ \mathrm{S},  \mathrm{SI}, \mathrm{SS}  \}$, for each $t \in \setTime{}$,
  \begin{align}
  \predictableVariation{M^\epsilon_{id}  }(t)  \ConvInProb {}& 0 \eqstop
  \end{align}

\label{lemma:ARJ2}
\end{myLemma}

\begin{proof}[Proof of Lemma~\ref{lemma:ARJ2}]
We proceed in the following two steps.

\noindent\paragraph{Computation of $\predictableVariation{ M_{\mathrm{S}}^\epsilon  }$}

Note that the original process $M'_{\mathrm{S}}$ 
makes only unit jumps. Then, for arbitrary $\epsilon >0 $, 
\begin{align*}
	\predictableVariation{  M_{\mathrm{S}}^\epsilon } (t)  \leq {}&  \int_{0}^{t} \Eof{ (\delta  M_{\mathrm{S}}^\epsilon(s) )^2  \indicator{ | \delta M'_{\mathrm{S}}(s) | > n^{1/2}  \epsilon   } \mid \history{s-}  }  \differential{s} ={}  0 \, \forall \, n > \frac{1}{\epsilon^2}  \\
	\implies {}   \predictableVariation{  M_{\mathrm{S}}^\epsilon } (t)  \ConvInProb {} & 0 \text{ for all }  0 < t \leq T \text{ and for all } \epsilon> 0 \text{ as } n \rightarrow \infty  \eqstop
\end{align*}

%
%

\noindent\paragraph{Computation of $\predictableVariation{ M_{\mathrm{SI}}^\epsilon  }$ and $\predictableVariation{ M_{\mathrm{SS}}^\epsilon  }$}

Note that both $M'_{\mathrm{SI}}$ and $M'_{\mathrm{SS}}$ 
 jump only if infection 
of a vertex occurs. 
This in particular implies that the jump sizes of $M'_{\mathrm{SI}}$ and $M'_{\mathrm{SS}}$  are bounded above by the degree of the vertex getting infected. 
Therefore, they are also bounded above by the maximum degree $d_{\mathrm{max}}$. For an arbitrary $\epsilon >0$, and for  ${id} \in \{ \mathrm{SI}, \mathrm{SS}  \}$,
\begin{align*}
	\predictableVariation{  M_{id}^\epsilon } (t)  \leq {}
&  \int_{0}^{t} \Eof{ (\delta  M_{id}^\epsilon(s) )^2   \indicator{ | n^{-1/2} d_{\mathrm{max}} | > \epsilon   }  \mid \history{s-}  } \differential{s} \\
	\leq {} & t n^{-1} d^2_{\mathrm{max}}   \indicator{ | n^{-1/2} d_{\mathrm{max}} | > \epsilon   }   \eqstop
\end{align*}
By Lemma~\ref{remark:max-degree}, along with  the continuous mapping theorem and the fact that almost sure convergence implies convergence in probability,
  we can claim that  the right-hand side of the above inequality $t n^{-1} d^2_{\mathrm{max}}   \indicator{ | n^{-1/2} d_{\mathrm{max}} | > \epsilon   }    \ConvInProb 0 $ for each $0 < t \leq T$ and $\epsilon >0$. Therefore, for all
$h >0$, $\probOf{   \predictableVariation{  M_{id}^\epsilon } (t) > h   } \leq \probOf{ t n^{-1} d^2_{\mathrm{max}}   \indicator{ | n^{-1/2} d_{\mathrm{max}} | > \epsilon   }   > h    }~\rightarrow~0 $   as   $n \rightarrow \infty$, establishing $   \predictableVariation{  M_{id}^\epsilon } (t)  \ConvInProb 0$ as $n \rightarrow \infty$ for all $0 < t \leq T$ and $\epsilon >0$.  This completes the proof.
\end{proof}

\subsection{Statement and Proof of the FCLT}\label{clt}
Having shown the convergence of all relevant quadratic variation processes, we are now ready to present the  functional central limit theorem. First we state that the function $V$ found in Lemma~\ref{lemma:deterministic-limit} is a positive semidefinite (psd) matrix-valued function on $\setTime{}$, with positive semidefinite increments. Set $V(0) \defeq \boldsymbol{0}$, the $3 \times 3$ null matrix, so that we can  treat $V(t)$ as a psd matrix-valued function on the entirety of $\setTime{0}$. Let us denote the collection of all such psd $3 \times 3$  matrix-valued functions on $\setTime{0}$ that has psd increments and that is $\boldsymbol{0}$ at time zero by $\mathcal{V}$. Given such a matrix-valued function $V  \in \mathcal{V} $, let $G$ be a continuous Gaussian vector martingale such that
$\predictableVariation{G} =\optionalVariation{G}=V$. Such a process always exists \cite[Chapter II, p.~83]{andersen1997statistical}.  In particular, $G(t)- G(s) \sim N(  \boldsymbol{0}, V(t)- V(s) )  $, the multivariate normal distribution  for $ 0 \leq s \leq t$. \newline
\begin{myTheorem}[Functional Central Limit Theorem]
		Consider the stochastic SI model described in Section~\ref{subsec:model}. Assume \ref{itm:as1}, \ref{itm:as2} and \ref{itm:as3} for a configuration model graph $\mathcal{G}(\psi, n)$. Consider, for $t \in \setTime{0}$, the fluctuation process
	\begin{align}
		\label{eq:fluctuation_defn}
		Y(t) \defeq \sqrt{n} ( n^{- 1} X(t) - x(t) )  \eqstop
	\end{align}
	Assume  $\lim_{n \rightarrow \infty} Y(0) = U(0) $, for some nonrandom $U(0)$.  Then, there exists a matrix-valued function $V  \in \mathcal{V} $ on $\setTime{0}$ such that
	\begin{align}
	Y   \ConvInDist U   \text{ in }  D^{(3)} \text{ as }  n \rightarrow \infty \eqcomma
	\end{align}
	where $U$ is a continuous Gaussian vector semimartingale satisfying
	\begin{align}
		U(t) = U(0) + G(t) +  \int_{0}^{t} \nabla \HOperator{x(s), \vartheta(s) }{x} U(s) \differential{s} \eqcomma
	\end{align}
	where    $\nabla \HOperator{x, \vartheta }{x}   \defeq (( \partial_j  \HOperator{x, \vartheta}{i}  ))$ for $i,j \in \{ \mathrm{S},  \mathrm{SI}, \mathrm{SS}  \}$ and $G$ is a continuous Gaussian vector martingale such that $\predictableVariation{G} =\optionalVariation{G}=V$, provided $V$ remains finite on the entirety of $\setTime{0}$ and $\nabla \HOperator{x(s), \vartheta(s) }{x} $ is continuous.

	\label{thm:FCLT}
\end{myTheorem}

\begin{proof}[Proof of Theorem~\ref{thm:FCLT}]
We first prove an FCLT for the martingale process $M$ defined in Equation~\eqref{eq:SIFCLT:M_defn}. We wish to apply Rebolledo's functional central limit theorem for local martingales on $M$. A version of the Rebolledo  theorem adequate for our purpose is provided in Appendix~\ref{sec:rebolledo}. Note that, in the light of  Doob-Meyer decomposition given in Equation~\eqref{eq:doob-decomposition}, $M$ is indeed a pure jump, zero-mean, locally square integrable, \cadlag martingale. After having established an FCLT for the  martingale process $M$, we prove convergence of the fluctuation process $Y$. It suffices to carry out the following three steps.

\noindent\paragraph{(Step I) Deterministic Limit of $\predictableVariation{M}$}
Let $ (x, \vartheta)  $ be the solution of Equation~\eqref{eq:SIFCLT:LLN_ODE}
with initial condition $x (0) = \alpha $ and $\vartheta (0)=1$, as given in Theorem~\ref{thm:LLN}. 
Then, by virtue of Lemma~\ref{lemma:deterministic-limit}, we conclude, for each $t \in \setTime{}$, $\predictableVariation{M}(t)  \ConvInProb V(t)$, where the matrix-valued function~$V$ is defined in Equation~\eqref{eq:Vmatrix}, and we set $V(0) \defeq \boldsymbol{0}$, the $3 \times 3$ null matrix.

 \noindent\paragraph{(Step II)  Asymptotic Rarefaction of Jumps}
 Let $\epsilon>0$ be arbitrary. Consider the vector  $M^\epsilon$ of square integrable martingales containing all jumps of components of $M(t)$ larger than~$\epsilon$ in absolute value for~$\epsilon>0$, as defined in Equation~\eqref{eq:Mepsilon-defn}. Then, by means of Lemma~\ref{lemma:ARJ2}, we conclude $\predictableVariation{M^\epsilon_{id}}(t)  \ConvInProb  0 $,
   for each $t \in \setTime{}$ and $id \in \{\mathrm{S}, \mathrm{SI}, \mathrm{SS}\}$. 

   Now let $G$ be the continuous Gaussian vector martingale such that  $\predictableVariation{G} =\optionalVariation{G}=V$. In the light of Rebolledo's theorem for locally square integrable martingales (see Appendix~\ref{sec:rebolledo} and also, \cite[Chapter II, p.~83]{andersen1997statistical}), Step I and Step II are sufficient to establish
   \begin{align}
  ( M(t_1),M(t_2),\ldots, M(t_l))   \ConvInDist  ( G(t_1),G(t_2),\ldots, G(t_l))  \text{ as } n \rightarrow \infty  \nonumber
   \end{align}
   for all $ t_1,t_2, \ldots, t_l \in \setTime{}$. Furthermore, since $\setTime{}$ is dense in $\setTime{0}$, we conclude $M   \ConvInDist G $  { in }  $D^{(3)}$ { as }  $n \rightarrow \infty$,
 and $ \predictableVariation{M}  $ and $ \optionalVariation{ M}  $ converge uniformly on compact subsets of $\setTime{0}$, in probability, to $V$.

 \noindent\paragraph{(Step III)  Convergence of the Fluctuation Process}
In keeping with the  Doob--Meyer decomposition given in Equation~\eqref{eq:doob-decomposition},
\begin{align}
	Y(t) = Y(0) + M(t) + \int_0^{t}  \sqrt{n} ( \frac{1}{n} \FOperator{\tilde{X}(s)}{X}  -  \HOperator{ x  (s)  , \vartheta   (s)  }{x }   ) \differential{s} \eqcomma  \nonumber
\end{align}
we  expect the following limit process
\begin{align}
	U(t) = U(0) + G(t) +  \int_{0}^{t} \nabla \HOperator{x(s), \vartheta(s) }{x}  U(s) \differential{s} \eqstop \label{eq:SIFCLT:limitU}
\end{align}
Indeed, define
\begin{align*}
	\Delta (t) \defeq & \int_{0}^{t}   \sqrt{n} \left( \frac{1}{n} \FOperator{\tilde{X}(s)}{X}  -  \HOperator{ x  (s)  , \vartheta   (s)  }{x }    - \frac{1}{  \sqrt{n}   }  \nabla \HOperator{x(s), \vartheta(s) }{x} \, Y(s)   \right)  \differential{s}  \\
	=  &  \int_{0}^{t}   \sqrt{n} \bigg( \frac{1}{n} \FOperator{\tilde{X}(s)}{X}  -  \HOperator{   \frac{1}{n}X(s)  , \theta   (s)  }{x }  + \HOperator{   \frac{1}{n}X(s)  , \theta   (s)  }{x }   \\ &{} -  \HOperator{ x  (s)  , \vartheta   (s)  }{x }
	 - \frac{1}{  \sqrt{n}   }  \nabla \HOperator{x(s), \vartheta(s) }{x} \, Y(s)   \bigg)  \differential{s} \eqstop
\end{align*}
Note that the strong law of large numbers  in Theorem~\ref{thm:LLN} establishes uniform convergence (in probability) of the  operators $ n^{-1} \FOperator{\tilde{X}(s)}{X}$ and $\HOperator{   \frac{1}{n}X(s)  , \theta   (s)  }{x }   $, and the latter operator is Lipschitz continuous on its domain (see \cite{jacobsen2016large}). In the light of Theorem~\ref{thm:LLN} and  \ref{itm:as3}, it follows from the Lipschitz continuity of various multinomial compensators $C_m^{k}$ introduced in the proof of Lemma~\ref{lemma:deterministic-limit} that $\lim_{n \rightarrow \infty } \sqrt{n} \left( \frac{1}{n} \FOperator{\tilde{X}(s)}{X}  -  \HOperator{   \frac{1}{n}X(s)  , \theta   (s)  }{x }  \right) = 0$.    Moreover, we have just shown $M   \ConvInDist G   \text{ in }  D^{(3)}$.  If $V$ remains finite on the entirety of $\setTime{0}$,  the matrix-valued function  $\nabla \HOperator{x(s), \vartheta(s) }{x} $ is continuous,  and $\lim_{n \rightarrow \infty} Y(0) = U(0) $, for some  nonrandom $U(0)$, then we have
 $\sup_{t \in \setTime{0}}  | \Delta (t) | \ConvInProb 0 $ following Theorem~\ref{thm:LLN}, and by application of the continuous mapping theorem, we conclude
 \begin{align}
 Y   \ConvInDist U   \text{ in }  D^{(3)} \text{ as }  n \rightarrow \infty \eqcomma \nonumber
 \end{align}
where the Gaussian semimartingale $U$ satisfies Equation~\eqref{eq:SIFCLT:limitU} with the Gaussian martingale $G$ being such that $\predictableVariation{G} =\optionalVariation{G}=V$. 
This completes the proof.
\end{proof}

\section{Applications}
\label{sec:applications}
Here, we consider some applications of our result. As we discuss these applications, we shall also  present some numerical and simulation results that are intended not only to provide  insights into the dynamics of the process, but also to serve as a verification of  our results.

\subsection{Percolation}
There is a connection between the stochastic SI model and the percolation theory known from statistical physics and developed  to study  the process of liquid  filtering (``percolating'')   through a porous medium. Classical equilibrium-mechanics studies its stationary behaviour and premises upon the axiom that  the  underlying quantum-mechanical laws are  designed so as to maximize the entropy. Stationary distribution of such a stochastic system is given by the Boltzmann ensemble. This classical treatment of the subject, however, does not explain the non-equilibrium behaviour of the dynamical system, \ie, when it is still in a transient phase. Consequently, the non-equilibrium behaviour of percolation has aroused much interest in  recent times. Some notable contributions include \cite{Hinrichsen20061,barato2009noneqPhase}. 
The standard treatment of percolation, both equilibrium and non-equilibrium, has been extended in another important direction concerning the structure of the porous medium. Traditionally it has been studied on lattices and grids. Of late, however, percolation on random graphs has also been considered 
 (\cite{baroni2015FPPrandomgraph,callaway2000network,van2010percolation}). Continuing  in this direction, we shall treat (non-equilibrium) percolation as a dynamical process on a configuration model random graph and study its behaviour over a finite time interval. 

One of the key  quantities of interest in the study of non-equilibrium percolation is the time evolution of the number of wetted sites (also called ``active'' vertices in the literature).   The correspondence between our stochastic SI model as described in Section~\ref{subsec:model} and non-equilibrium percolation  is visible  if we  treat the infected vertices as the ones wetted during the process of percolation. Accordingly, in this context, we give the process $X(t)$ appropriate new interpretation. The process $X_{\mathrm{S}}(t)$, for example, captures the number of  unwetted sites until time~$t$, and the process $X_{\mathrm{SI}}$, the number of channels (bonds) through which the liquid can percolate. In Figure~\ref{fig:percolation}, the  percolated component up to a given time (the wetted part of the graph) is shown in red. Having made the correspondence precise, we can apply Theorem~\ref{thm:FCLT} to approximate these quantities in the large graph limit.



 \begin{figure}[t!]
 	\centering
 	\begin{subfigure}[t]{\columnwidth}
 			\centering
 		\includegraphics[width=0.8\columnwidth]{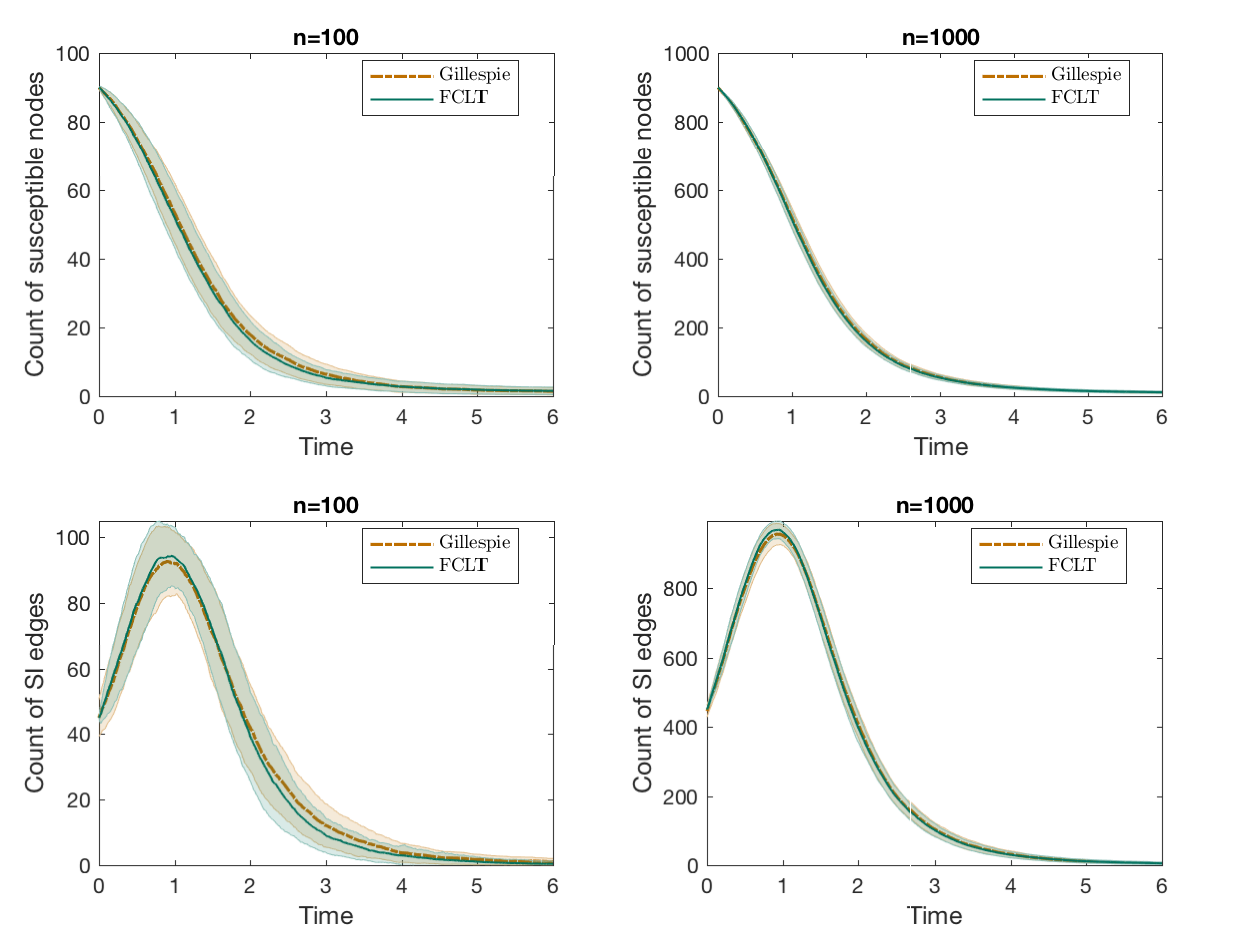}
 		\caption{Simulation setting: Poisson distribution with $\lambda = 5, \alpha_{\mathrm{S}}=0.9$, and $\beta=0.5$.}
 		\label{fig:Poisson_Errorbar_S_SI}
 	\end{subfigure}
 	\begin{subfigure}[t]{\columnwidth}
 		\centering
 		\includegraphics[width=0.8\columnwidth]{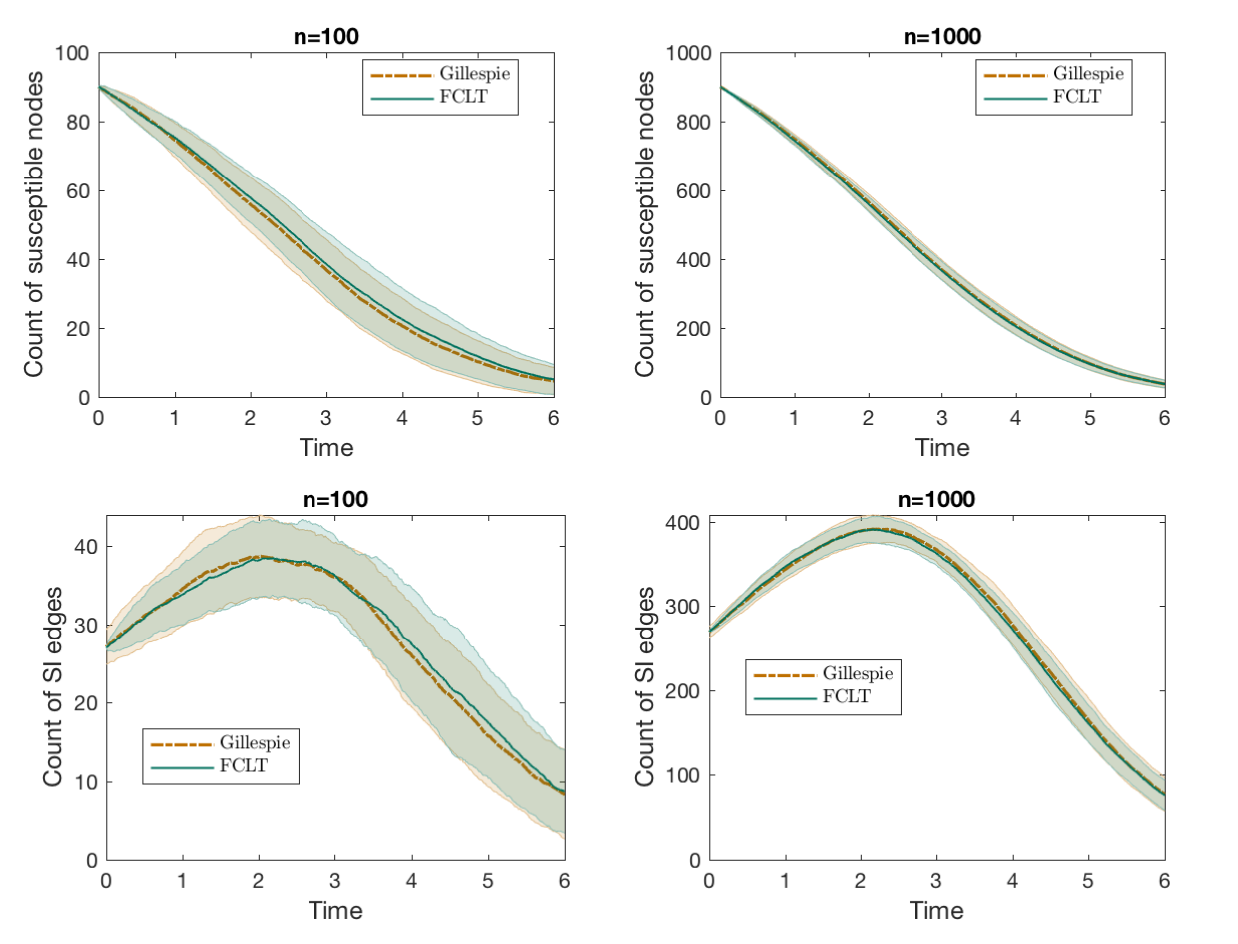}
 		\caption{Simulation setting: $r$-regular random graph with $r=3, \alpha_{\mathrm{S}}=0.9$, and $\beta=0.5$.}
 		\label{fig:RegularRandomGraph_Errorbar_S_SI}
 	\end{subfigure}
 	\caption{\label{fig:errorbars}%
 	Comparison of our diffusion approximation with simulation results obtained by Gillespie's algorithm. 
 	}

 \end{figure}

 \begin{figure}[t!]
 	\centering
 	\begin{subfigure}[t]{0.49\textwidth}
 	\includegraphics[width=\columnwidth]{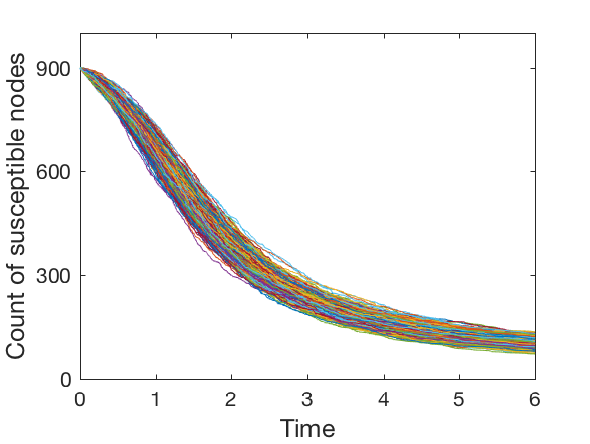}
 	\end{subfigure}
 	\begin{subfigure}[b]{0.49\textwidth}
 	\includegraphics[width=\columnwidth]{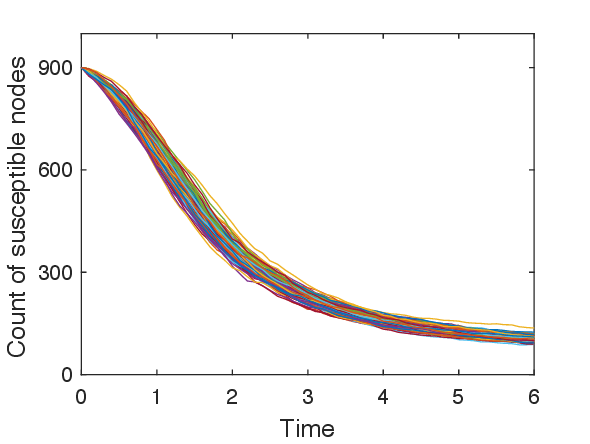}
 	\end{subfigure}

 	\begin{subfigure}[t]{0.49\textwidth}
 	\includegraphics[width=\columnwidth]{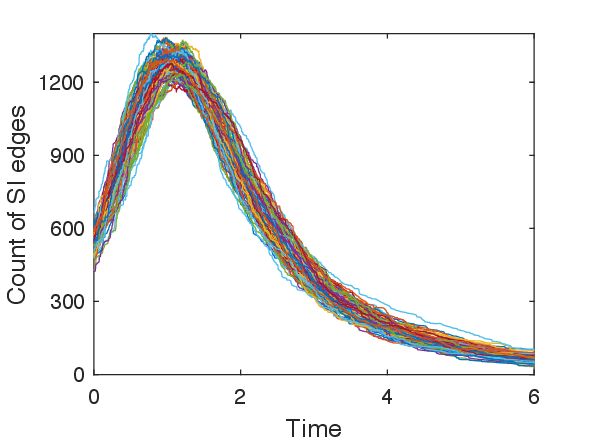}
     \end{subfigure}
 	\begin{subfigure}[b]{0.49\textwidth}
 	\includegraphics[width=\columnwidth]{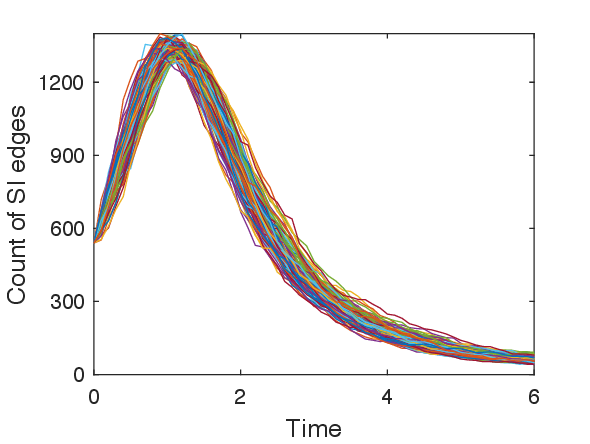}
     \end{subfigure}
 	\caption{\label{fig:NB_comparison_n10000} %
 		Comparison of simulated sample paths. \textbf{(Left)} Sample paths obtained through Gillespie's algorithm. \textbf{(Right)} Sample paths obtained through diffusion approximation. Simulation setting: $n=1000$, Negative Binomial distribution with $r=2, p=3/4$.}
 \end{figure}

\paragraph{Numerical Illustration}
In Figure~\ref{fig:errorbars}, we show some simulation results to check the accuracy of our scaling limit. 
We compare the expected sample paths of $X_{\mathrm{S}}$ and $X_{\mathrm{SI}}$ provided by Theorems~\ref{thm:LLN} and \ref{thm:FCLT}, with estimates
 obtained using simulations of the Gillespie's algorithm on  a CM graph. In particular, we considered a Poisson degree distribution in  Figure~\ref{fig:Poisson_Errorbar_S_SI} and a  $3$-regular random graph in Figure~\ref{fig:RegularRandomGraph_Errorbar_S_SI} (obtained by the CM construction with degree distribution~$p_k=\indicator{k=3}$). In Figure~\ref{fig:NB_comparison_n10000}, we compare the simulated sample paths of the true Gillespie dynamics and that corresponding to the diffusion approximation for Negative Binomial degree distribution. Figures~\ref{fig:Poisson_Errorbar_S_SI},  \ref{fig:RegularRandomGraph_Errorbar_S_SI} and \ref{fig:NB_comparison_n10000} show convincing accuracy of the diffusion approximation. In Figure~\ref{fig:ConfEllipse}, we show the time evolution of the correlation coefficient between the jumps of $X_{\mathrm{S}}$ and $X_{\mathrm{SI}}$, and also the expected sample path coupled with $95\%$-confidence ellipses in the space of  $X_{\mathrm{S}}$ and $X_{\mathrm{SI}}$. The orientation of the confidence ellipses is calculated as the angle of the eigenvector corresponding to the largest eigenvalue of the covariance matrix towards the $x$-axis. To be specific: the orientation is given by $\omega \defeq \arctan{e_2/e_1}$, where $e \defeq (e_1, e_2)$ is the eigenvector corresponding to the largest eigenvalue of the covariance matrix. The lengths of the major and the minor axes are determined using the eigenvalues and by looking up the probability table of chi squared distribution (recall that squared normal variates follow a chi squared distribution). The Matlab script used to draw the ellipses is based on a script provided by \cite{confidenceEllipses}.


   \begin{figure}[h!]
   	\centering
   	\begin{subfigure}[b]{1.0\columnwidth}
   		\includegraphics[width=\columnwidth]{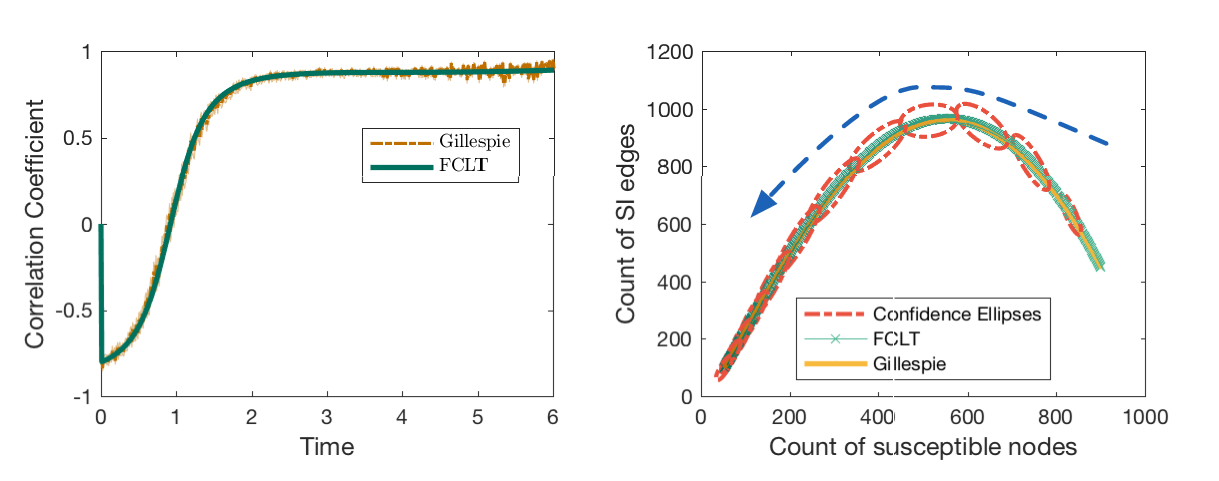}
   		\label{fig:ConfEllipsoidsPoiss}
   	\end{subfigure}
   	\hfill
   	\begin{subfigure}[b]{1.0\columnwidth}
   		\centering
   		\includegraphics[width=\columnwidth]{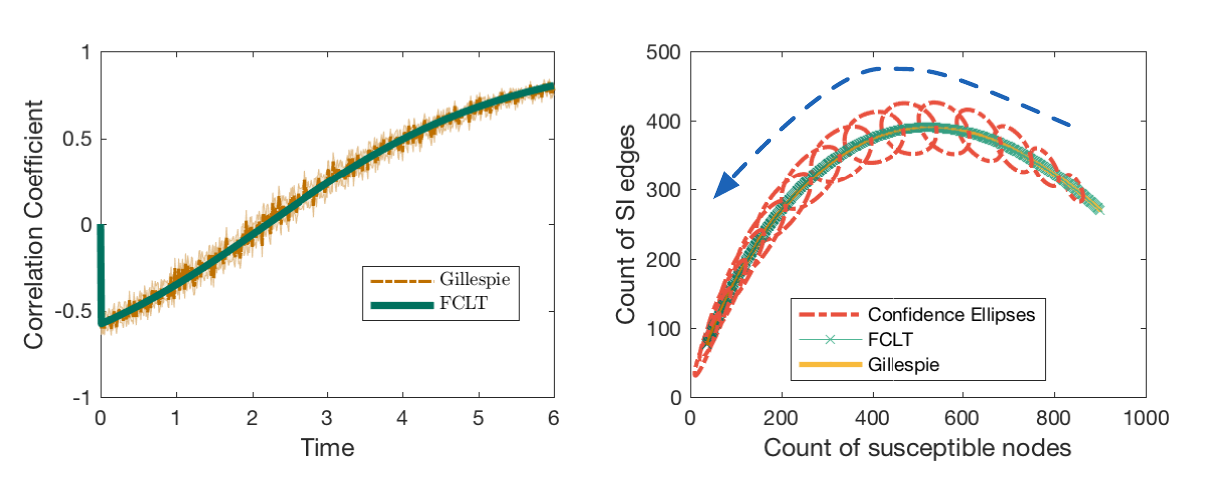}
   		\label{fig:ConfEllipsoidsRandReg}
   	\end{subfigure}
   	\caption{The figures on the left depict the time evolution of the correlation coefficient between jumps of $X_{\mathrm{S}}$ and $X_{\mathrm{SI}}$ as estimated from numerical simulations (via Gillespie's algorithm) pitted against theoretical values computed from the functional central limit theorem (Theorem~\ref{thm:FCLT}). The figures on the right show the expected sample path in the space of $X_{\mathrm{S}}$ and $X_{\mathrm{SI}}$. The two lines
   		correspond to numerical simulation and theoretical values. 
   		The dotted ellipses are the $95\%$-confidence ellipses based on  estimates of covariances between $X_{\mathrm{S}}$ and $X_{\mathrm{SI}}$ from the diffusion approximation created using a Matlab script provided by \cite{confidenceEllipses}.  The arrows indicate the time direction.  \textbf{(Above)} Poisson distribution with mean $5$. \textbf{(Below)} $r$-regular random graph with $r=3$. In both cases, $n=1000,  \alpha_{\mathrm{S}}=0.9$,  and $\beta=0.5$. }
   	\label{fig:ConfEllipse}
   \end{figure}

\begin{figure}[h!]
	\centering
	\includegraphics[width=1.0\columnwidth]{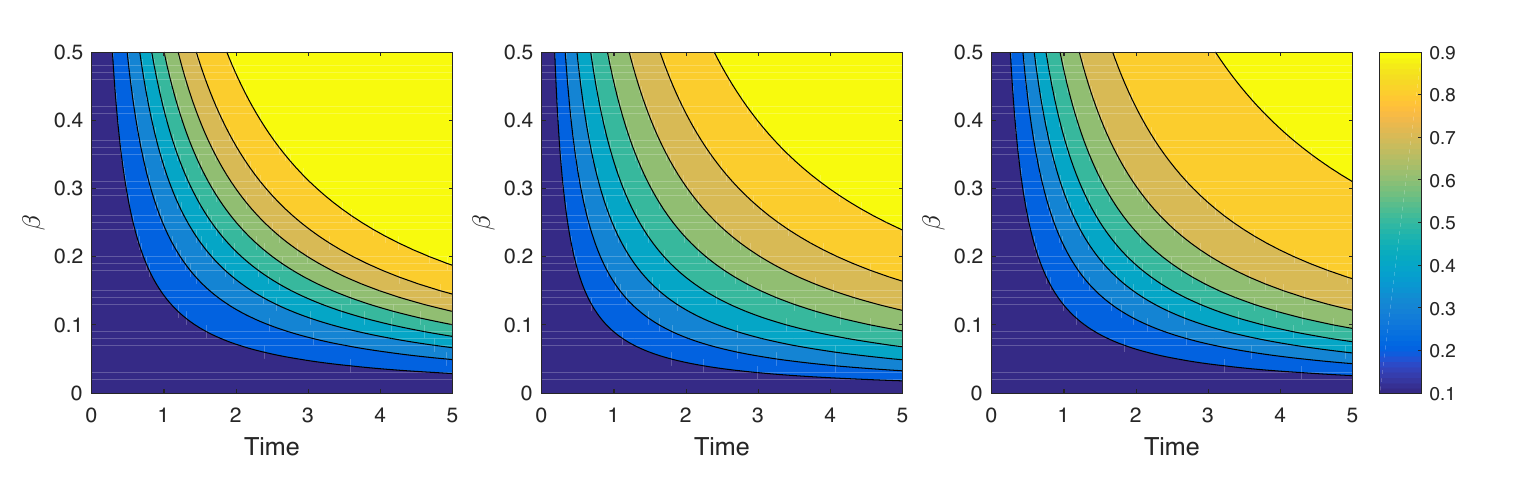}
	\caption{Comparison of percolation profiles of three degree distributions having the same mean. \textbf{(Left)} Poisson distribution with mean $6$. \textbf{(Middle)} A heterogeneous population with degree distribution $p_k \defeq  0.7 \times \indicator{k=1}  +  0.2 \times \indicator{k=4} +0.1 \times \indicator{k=45}$. Such a degree distribution  represents a population segregated into three classes. Weak vertices constitute the biggest class, followed by  medium strength vertices and then strong vertices. \textbf{(Right)} Negative Binomial distribution with parameters $r=2, p={3}/{4} $. The figures show time evolution of the fraction of vertices on the percolated component for varying infection rates~$\beta$. We assume the initial fraction of infected vertices is $0.1$ in all three cases. The yellow region in each of the plots corresponds to the terminal state. Questions such as whether the system with an infection rate~$\beta$ ``percolates'' are immediately settled by drawing a horizontal line and checking whether the lines passes through the colour corresponding to a pre-specified level. It is worth noting that the limiting percentages of vertices in the giant components eventually getting infected are also quite different for the three degree distributions. While it is around $0.91$ (approximately) for the Negative Binomial distribution and the hand-picked degree distribution in the middle, the percentage is relatively high for the Poisson distribution (around $0.99$ approximately).  }
	\label{fig:Isolines}
\end{figure}

The existence of a giant component and the proportion of vertices on the giant component play an important role in percolation theory, especially from an equilibrium point of view in statistical mechanics. The case of a degree distribution~$\{p_k \}_{k \in \setOfNonnegativeIntegers}$ such that $\sum_{k \in \setOfNonnegativeIntegers} k^2 p_k = 2 \sum_{k \in \setOfNonnegativeIntegers} k p_k$ and $p_1=0$ (or, equivalently $p_0 + p_2 = 1$)  is a curious one in that  quite different behaviours of the giant component are observable for such a degree distribution. Please refer to \cite{Janson2009newTechnique} for examples of such behaviours. Barring this exceptional case, in
 the light of \ref{itm:as3}, the condition for existence of a giant component is  satisfied (see  \cite{Molloy1995critical}) for our stochastic SI model in the traditional sense. To be precise, setting $\alpha_{\mathrm{S}}=1$ and taking  asymptotic limit in time, one finds the fraction of vertices on the giant component to be $1- \psi(\theta_\infty)$, where $\theta_\infty>0$ is the solution of $\partial\psi(1) \theta_\infty= \partial \psi( \theta_\infty)$ (see \cite{janson2014law,molloy1998giant}). However, as mentioned earlier, we take a non-equilibrium point of view and concern ourselves with the time evolution of the fraction of vertices on the infected part of the graph, the ``percolated component'' . As a by-product of the scaling limits in Theorem~\ref{thm:LLN} and Theorem~\ref{thm:FCLT}, the variable $\theta$ defined in Equation~\eqref{eq:theta_defn} gives us a tool to approximate the proportion of susceptible individuals in the population (and hence, the proportion of infected vertices as well). We expect the fraction of infected individuals to converge in probability to $1- \alpha_{\mathrm{S}} \psi(\vartheta)$ as $n \rightarrow \infty$, $\vartheta$ being the scaling limit of $\theta$.
A fixed time interval~$\setTime{0}$ enables us to look for critical values in the space of the infection rate~$\beta >0$. This allows us to  decide whether the system ``percolates'' in the sense that the fraction of vertices on the percolated component achieves a value greater than a pre-specified one (usually  close to unity) by time~$T$.
Using different colours in Figure~\ref{fig:Isolines}, we depict  the  fraction of vertices on the percolated component as a function of both time and the infection rate~$\beta$ (let us call such a figure a percolation profile) for three degree distributions with the same mean. Questions such as whether the system with an infection rate~$\beta$ ``percolates'' are immediately settled by drawing a horizontal line and checking whether the line passes through the colour corresponding to the pre-specified level. See Figure~\ref{fig:perc_degree_cost} for another comparative view, highlighting the need to take into account higher moments of the degree distribution.

\begin{figure}[h!]
	\centering
	\includegraphics[width=\columnwidth]{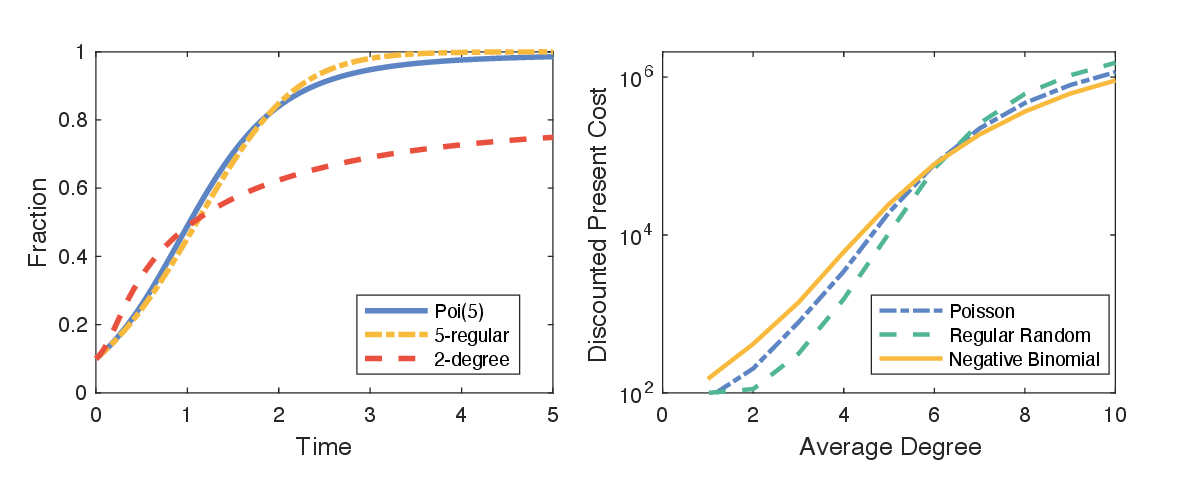}
	\caption{\textbf{(Left)} Comparison of the time evolution of  the fraction of vertices on the percolated component for different degree distributions with the same mean.
 The $2$-degree distribution in the plot refers to the degree distribution $p_k = 0.5 \times  \indicator{k=1} + 0.5 \times  \indicator{k=9}$, where none of the vertices have degree~$5$ yet the average degree is $5$. This presents a pathological case and highlights the need to take into account higher moments of the degree distribution. \textbf{(Right)} Comparison of discounted cost against increasing average degree for different degree distributions. With increasing average degree the graphs lose sparsity and facilitate spread of computer virus. Therefore, they incur higher cost. When the average degree is very small, regular random graphs seem favourable compared to random graphs with negative binomial distributed degrees. The costs are computed with $n=1000, \gamma=1$.}
	\label{fig:perc_degree_cost}
\end{figure}

\subsection{The strange case of Poisson-type degree distributions and the exactness of the pair approximation}
 We consider a  particular class of degree distributions called  ``Poisson-type'' (PT) by \cite{jacobsen2016large}. A degree distribution with PGF $\psi$ is called PT if $\kappa(\vartheta )$, defined in Equation~\eqref{eq:kappa_defn},  is a constant, \ie, $\kappa (\vartheta) =\kappa$ for some $\kappa \in \setOfReals$ (or equivalently, $\partial{\psi (\vartheta)}= \partial{\psi (1)} \left(\psi (\vartheta)\right)^{\kappa }$). As a consequence, the operators defined in Equation~\eqref{eq:diffOperatorDefn} are also constants, and satisfy
\begin{align}
	\diffOperator{\psi\circ \vartheta}{r} = \prod_{i=1}^{r-1} \left(  i \kappa - i +1  \right)  = \left( (r-1)\kappa -r  +2 \right) \, \diffOperator{\psi\circ\vartheta}{r-1}   \eqcomma \nonumber
\end{align}
with  $\diffOperator{\psi\circ \vartheta}{0} =1$. The PT class includes Poisson ($\kappa=1$, irrespective of the mean of the distribution), degenerate distribution ($r$-regular random graphs, $\kappa =\frac{r-1}{r} <1 $), binomial ($\kappa = \frac{N-1}{N} <1$, independent of $p$ for $\mathrm{Bin}(N,p)$), negative binomial ($\kappa = \frac{r+1}{r}>1$, independent of $p$ for $\mathrm{NB}(r,p)$) degree distributions. The PT class  is particularly peculiar in that it totally decouples the vector $x = (x_{\mathrm{S}}, x_{\mathrm{SI}}, x_{\mathrm{SS}})$, and the matrix-valued function~$V$  from the auxiliary variable $\vartheta$ so that an autonomous system of  ODEs can be obtained for $x$ and $V$, rendering $\vartheta$   redundant. This allows for great simplification in the limiting equations. Define  $\GOperator{x}{} \defeq (\GOperator{x }{S},\GOperator{x}{SI},\GOperator{x}{SS} )$  as
\begin{align}
\begin{aligned}
\GOperator{x}{S} & \defeq  -\beta x_{\mathrm{SI}}   \eqcomma \\
\GOperator{x }{SI} & \defeq   \beta   \kappa \frac{x_{\mathrm{SI}}}{x_{\mathrm{S}}} (x_{\mathrm{SS}}- x_{\mathrm{SI}}) - \beta  x_{\mathrm{SI}}   \eqcomma \\
\GOperator{x }{SS} & \defeq  -2 \beta \kappa  \frac{ x_{\mathrm{SI}} x_{\mathrm{SS}} }{x_{\mathrm{S}}} \eqstop
\end{aligned}
\label{eq:pair_approx_operators_defn}
\end{align}
Plugging  $\diffOperator{\psi\circ \vartheta}{2} =\kappa$, and $\diffOperator{\psi\circ\vartheta}{3} =\kappa (2 \kappa -1) $ in Equation~\eqref{eq:vdefinition}, the matrix-valued function $V$ is  entirely determined by $x$. The following  is immediate.

\begin{MyCorollary}[Scaling limit under PT distributions]
Assume \ref{itm:as1}, \ref{itm:as2}, and  \ref{itm:as3} for a configuration model graph $\mathcal{G}(\psi, n)$ with
  $\partial\psi (\vartheta) =  \partial\psi (1) (\psi (\vartheta))^\kappa $ for some $\kappa \in \setOfReals$. Then, the following law of large numbers holds
  \begin{align*}
  	\sup_{0<t \leq T}  \norm{ n^{-1} X(t) -x(t)  } \ConvInProb 0 \eqcomma
  \end{align*}
where $x$ is the solution of $ x(t) = x(0) + \int_{0}^{t} \GOperator{x(s)}{} \differential{s} $ with $x(0) = \alpha$. Moreover, the fluctuation process $Y$ defined in Equation~\eqref{eq:fluctuation_defn} converges weakly to a continuous Gaussian vector semimartingale $U$ satisfying
\begin{align*}
	U(t) = U(0) + G(t) +  \int_{0}^{t}  \nabla \GOperator{x(s)}{} U(s) \differential{s}  \eqcomma
\end{align*}
where $G$ is a  Gaussian vector martingale such that $\predictableVariation{G} = \optionalVariation{G} = V$.
\label{thm:PT_Limits}
\end{MyCorollary}

In fact, one can obtain a smaller system by expressing  $x_{\mathrm{SI}}$ and $x_{\mathrm{SS}}$ explicitly as a function of $x_{\mathrm{S}}$ (see \cite{jacobsen2016large}). This is remarkable because, under the PT class, the graph structure impacts the scaling limits \emph{only} through two summary statistics of $\psi$, namely the mean $\partial \psi(1)$ and $\kappa=\diffOperator{\psi(1)}{2} $. Recall that
$\kappa$, as defined in Equation~\ref{eq:kappa_defn}, is the limiting ratio of the average excess degree of a susceptible vertex chosen at random as a neighbour of an infected vertex, to the average degree of a susceptible vertex. Therefore it is, in general, dependent on time through $\vartheta$. Under the PT class, this ratio remains constant throughout the entire course of time~$\setTime{0}$. Moreover, the mean  $\partial \psi(1)$ only impacts the initial condition $x(0)=\alpha$ through Equation~\eqref{eq:alpha_defns}. The dynamics of the limiting process are then dictated by the constant $\kappa$ under the PT class.

Now we revisit the correlation equations approach of \cite{Rand2009Correlation} from ecology literature to study the dynamics of  counts of singles, pairs, triples, and quadruples of the form $A, AB, ABC, ABCD$, where $A,B,C,D \in \{ \mathrm{S}, \mathrm{I}\}$.  Following \cite{Rand2009Correlation}, we use the notation  $\concentration{\cdot}$ to denote the count.   In this  mean-field approach, the dynamics of singles are described by that of pairs; dynamics  of pairs, by triples, and so on. In this context, pair approximation refers to approximating the count of triples by pairs in the following way
\begin{align*}
\concentration{ABC} \approx \kappa \frac{\concentration{AB} \concentration{BC} }{\concentration{B}} \eqcomma
\end{align*}
and closing the system at the level of pairs (also known as pair closure). In order to draw an analogy, we divide the counts by $n$, and use the same notation for the scaled counts. We also set  the same initial condition  $(\concentration{\mathrm{S}},\concentration{\mathrm{SI}},\concentration{\mathrm{SS} }) = \alpha$ at $t = 0$. The pair approximation then yields a system of ODEs for $(\concentration{\mathrm{S}},\concentration{\mathrm{SI}},\concentration{\mathrm{SS}})$ that {exactly} matches the limiting ODEs for $n^{-1}X$, \ie,

\begin{align}
	\label{eq:pair_approx}
\frac{\differential}{  \differential{t} }	   { (\concentration{\mathrm{S}}, \concentration{\mathrm{SI}}, \concentration{\mathrm{SS}})  } = \GOperator{ (\concentration{\mathrm{S}}, \concentration{\mathrm{SI}}, \concentration{\mathrm{SS}})  }{}  \eqstop
\end{align}

Therefore, under the PT class, the pair approximation is \emph{exact} in the sense that it correctly estimates the limiting means of various counts.  By virtue of  Corollary~\ref{thm:PT_Limits},  our FCLT further enables it to correctly estimate \emph{all} other higher limiting moments, because $V$ is now entirely determined by the solution of Equation~\eqref{eq:pair_approx}. As the PT class is quite big, our FCLT thus greatly enhances the usefulness of the pair approximation.

\subsection{Spread of Computer Viruses}
Epidemic models have been used in the context of spread of computer virus for some years now. The correspondence between our model and the application area under consideration is apparent without requiring much change in nomenclature. Early works in this direction did not take into account the inherent graph structure and assumed ``homogeneous mixing'' in some sense. Recent works, however, duly studied it on more realistic computer networks, which are often modelled as random graphs, without the assumption of ``homogeneous mixing''. Lelarge, for example, based much of his work on classical Erd\H{o}s   R\'enyi  random graphs  and configuration models (see, \eg, \cite{lelarge2012diffusion}).  Interested readers are referred to
 \cite{Kephart1993virus,lelarge2012diffusion,Wierman2004virus}   for an overview of relevant literature.  Applying our results, we can approximate the number of virus-affected computers over time and the edges of different types. Additionally, one might be  interested in estimating some ``cost''  involving the count variables in a linear or non-linear fashion. For instance, if the cost function is polynomial in the count variables, the mixed moments of various orders 
 can be approximated by means of Theorem~\ref{thm:FCLT}.
To illustrate the concept using a simple example, we assume an exponentiated form for the  incurred cost  to emphasize the severity of a computer being virus-affected. We can then compute  time-discounted expected incurred cost and study how it behaves with decreasing sparsity of the underlying graph. To this end,  define
\begin{equation}
\begin{aligned}
I(t)  \defeq {} & \myExp{c X_I (t)  }   \eqcomma  \\
C_\psi    \defeq {} & \Eof{   \int_{\setTime{0}  } \myExp{ -  \gamma  t }   I(t)  \differential{t} }  = \int_{\setTime{0}  } \myExp{ - \gamma t } \Eof{ I(t)  }  \differential{t}  \eqcomma
\end{aligned}
\label{eq:costdefn}
\end{equation}
 where $c>0$ and $\gamma>0$ are constants. In Figure~\ref{fig:perc_degree_cost}, we plot the discounted cost~$C_\psi $ against an increasing average degree of the underlying graph, engendering decreasing sparsity. When the average degree is very small, regular random graphs seem favourable compared to random graphs with negative binomial distributed degrees.

\section{Conclusion and Future Work}
\label{sec:conclusion}
We conclude the paper with a brief literature review and  a short discussion afterwards. In summary, we study the susceptible-infected (SI) model by formulating a stochastic process on configuration model random graphs. Even though this 
is a simple infection model, the Markovian process on the entirety of the random graph suffers state space explosion as $n$ grows to infinity. Analysis of the non-Markovian aggregate process also becomes complicated. Therefore, scaling limits retaining key features of the network are generally  of interest.

\subsection{Related Works}
\label{subsec:relatedWork}
In the recent scientific literature one  comes across a host of dynamical processes arising from epidemiology (\cite{brauer2003mathEpidemiology,newman2002epidemicOnNetworks,House2014dynamics}), 
statistical physics (\cite{bollobas2006percolation}), and computer science (\cite{lelarge2012diffusion,Kephart1993virus,Wierman2004virus}).  These dynamical processes are often similar and hence, lend themselves to application across disciplines  (\cite{pastor2001epidemicB}). In pursuit of scaling limits, much of the research 
has been inspired by the mean-field approach from statistical physics. 
For instance, the authors in \cite{pastor2001epidemicB} 
study  epidemic dynamics on scale-free networks of \cite{barabasi1999emergence}. The majority of work in this direction aims to obtain limiting  ordinary differential equations (ODEs) for the proportions of individuals in different compartments of the population.
Notwithstanding the simplicity of these methods,  the scaling limits presented are  approximate and lack mathematical rigour  by design. See \cite[Chapter 1]{durrett2007randomGraphDynamics} for a critique.
The standard mean-field method was further improved by use of pair-approximation in \cite{gleeson2011pair}. Several other improvements  yielding less approximate results have been proposed afterwards. A detailed account is presented in \cite{vespignani2008dynamicalProcesses}. 
Some of these approximate results have been followed up by probabilists and  improved upon (\cite{durrett2007randomGraphDynamics,durrett2010some,chatterjee2009contact}).

A related line of research concerns the first-passage percolation (FPP) on random graphs (\cite{baroni2015FPPrandomgraph,callaway2000network,van2010percolation,Bhamidi2017FPPSparse,Bhamidi2010FPPmean,Bhamidi2011FPPonER}). Given a graph $G = (V, E)$ with $n$ vertices, we assign random weights $W_e$ to  edges $e \in E$. We assume the weights are independent and identically distributed. For two vertices $i, j \in V$, the \emph{passage time} from $i$ to $j$ is defined as
\begin{align*}
\mathsf{c}_{n}(i,j) \defeq \min_{\mathsf{p}(i,j) } \sum_{e \in \mathsf{p}(i,j) } W_e \eqcomma
\end{align*}
where the minimum is taken over all paths $\mathsf{p}(i,j) $  from the vertex $i$ to the vertex $j$. By convention, $\mathsf{c}_{n}(i,i) \defeq 0$, and $\mathsf{c}_{n}(i,j) \defeq \infty$ if there is no path from the vertex $i$ to the vertex $j$. For two typical vertices $i$ and $j$, chosen uniformly at random, one  looks for a sequence of reals $ \tilde{\mathsf{c}}_{n}$ such that
\begin{align*}
 \left(  \mathsf{c}_{n} - \tilde{\mathsf{c}}_{n} \right) \ConvInDist Q \eqcomma
\end{align*}
for some limiting random variable $Q$ with usually continuous support. Sometimes scaling limit of the number of edges on the shortest path between two typical vertices is also studied along with the passage times. That is, if $\mathsf{h}_{n}$ denotes the number of edges on the shortest path between two typical vertices (chosen uniformly at random), we seek a  sequence $\tilde{\mathsf{h}}_{n}$ of scaling constants such that
\begin{align*}
\left(  \mathsf{h}_{n} - \tilde{\mathsf{h}}_{n} \right) \ConvInDist \tilde{Q} \eqcomma
\end{align*}
for some limiting random variable $\tilde{Q}$. The random variable $\mathsf{h}_{n}$ is often called the typical hop-count.  One of the important questions in the study of FPP is regarding the growth rate of the sequence of the scaling constants $\tilde{\mathsf{c}}_{n}$, and $\tilde{\mathsf{h}}_{n}$, \eg, whether they are of order $\log (n)$. Universality results are also important in the study of FPPs and form a substantial body of literature. While the hop-count $ \mathsf{h}_{n} $ measures typical distances in the random graph, the passage time $\mathsf{c}_{n}$ can be interpreted as the typical (minimum) amount of time required for an infectious disease to transmit from an infected vertex to a susceptible vertex.  As a result, limit theorems for the passage time $\mathsf{c}_{n}$  provide a complementary view to our FCLT for the SI process on CM random graphs as demonstrated by the numerical results in Section~\ref{sec:applications}.

In an epidemiological context,  limit theorems for a discrete-time random graph epidemic model were derived in  \cite{hakan_andersson1998} under rather restrictive assumptions 
such as finiteness of a $(4+\delta)$-th moment of the degree distribution, for some $\delta >0$.   The work of Erik Volz in
\cite{Volz2008} 
presented scaling limits  for susceptible-infected-removed (SIR) model on random graphs in the form of ODEs. 
The authors in \cite{decreusefond2012large} later proved 
Volz's  results rigorously by summarising the epidemic process on configuration model random graphs into some measure-valued equations. Several similar  laws of large numbers-type scaling limits under varying sets of technical assumptions surfaced afterwards. For example,  uniformly bounded degrees were assumed in \cite{bohman2012sir,barbour2013approximating}. The authors in \cite{janson2014law} assume degree of a randomly chosen susceptible vertex to be uniformly integrable and the maximum degree of  initially infected vertices to be $\smallO{n}$. 
The work in \cite{BALL2002SIR} studies a variant of the standard compartmental SIR with notions of local (within households, for example) and global contacts, and uses a branching process approximation to derive threshold behaviour and final outcome in the event of a global epidemic.
 Recently a  law of large numbers 
for the stochastic SIR process on a multilayer configuration model was derived in  \cite{jacobsen2016large}   assuming finiteness of the second moment of the underlying degree distribution.

%
%


 Although a number of scaling limits in the form of laws of large numbers have surfaced over the years, appropriate diffusion approximations are not yet fully explored. 
 Our FCLT attempts to 
 complement the laws of large numbers  already available in the literature.
In particular, our FCLT lends itself as an approximating tool in applications where laws of large numbers are inadequate, \eg, in situations involving higher order moments.

In our present work, we have disregarded ``recovery'' of  the infected vertices. The reason behind this exclusion is our inability to evaluate the neighbourhood distribution of an infected vertex in the presence of spontaneous recovery of its neighbours.
One difficulty is that, unlike the susceptible vertices (of a given degree) that are  untouched by the process of infection and hence, receive identically distributed neighbourhoods upon uniformly-at-random matching of half-edges, the infected vertices are not identically distributed because they already possess partially formed neighbourhoods consisting of infected and recovered neighbours. This corresponds to the part of the graph that has already been revealed up to a given time. Recall the  construction of the configuration model random graph where the graph is  dynamically revealed as  infection spreads (see Section~\ref{sec:FCLT}). As a result, the hypergeometric argument as mentioned in Lemma~\ref{remark:susceptibleNbd}  seems  inadequate. For the purpose of obtaining a law of large numbers, we can circumvent this difficulty by suitably bounding the jump sizes of different martingales arising in the proof  by the degrees of the vertices concerned. Therefore, we actually do not need the exact neighbourhood distribution of an infected individual for deriving laws of large numbers. However, to establish an FCLT, one needs to find the limit of the quadratic covariation process that would involve the task of approximating quantities such as $\sum_{k \in \setOfNonnegativeIntegers} \sum_{i \in I_k} X_{\mathrm{IS},i}^2$, where $I_k$ is the collection of degree-$k$ vertices that are infected and $X_{\mathrm{IS},i}$ is the number of 
susceptible neighbours of an infected individual of degree~$k$. We suspect an elaborate bookkeeping of the infection spreading process would be necessary to approximate such quantities. We have not been able to find a simple workaround so far and intend to pursue this problem in the near future.

\appendix

%
\section{Hypergeometric  Moments}
\label{sec:hypergeometric-moments}
Here, we compute various (conditional) moments that are useful for our derivations. Let us use the shorthand notation $$ \varsigma  (n_{\mathrm{SI}} ; X_{\mathrm{SI}}, X_{\mathrm{S} \Bigcdot}, k   ) \defeq 	\probOf{   X_{\mathrm{SI},i} = n_{\mathrm{SI}}, X_{\mathrm{SS} ,i}= k- n_{\mathrm{SI}} \mid \history{t} } \eqcomma  $$
conditional on the process history as given in Lemma~\ref{remark:susceptibleNbd}. The following moments are then computed keeping  Lemma~\ref{remark:susceptibleNbd} in mind.

In a straightforward fashion, we get for $i \in S_k$,
\begin{align*}
	\Eof{    \NPermuteR{X_{\mathrm{SI},i}}{3} \mid \history{t} } ={} & \sum_{n_{\mathrm{SI}}} \NPermuteR{n_{\mathrm{SI}}}{3} \frac{  \binom{X_{\mathrm{SI}}}{n_{\mathrm{SI}}} \binom{ X_{\mathrm{S} \Bigcdot } - X_{\mathrm{SI}}}{k-n_{\mathrm{SI}}  } }{   \binom{X_{\mathrm{S} \Bigcdot } }{k} }   \\
	={} & \sum_{n_{\mathrm{SI}} } \NPermuteR{n_{\mathrm{SI}}}{3} \frac{    \frac{ \NPermuteR{X_{\mathrm{SI}}}{3} }{  \NPermuteR{n_{\mathrm{SI}}}{3}  }  \binom{X_{\mathrm{SI}} -3 }{n_{\mathrm{SI}}-3 }   \binom{ (X_{\mathrm{S} \Bigcdot } -3) - (X_{\mathrm{SI}}  -3) }{  (k-3) -(n_{\mathrm{SI}}-3)  } }{  \frac{ \NPermuteR{X_{\mathrm{S} \Bigcdot }}{3} }{  \NPermuteR{k}{3}  } \binom{X_{\mathrm{S} \Bigcdot } -3 }{k-3} }    \\
	={} & \frac{   \NPermuteR{k}{3} \NPermuteR{X_{\mathrm{SI}}}{3}  }{\NPermuteR{X_{\mathrm{S} \Bigcdot }}{3} }  \sum_{n_{\mathrm{SI}}}  \varsigma  (n_{\mathrm{SI}} ; X_{\mathrm{SI}}-3,  X_{\mathrm{S} \Bigcdot}-3, k -3  )  \\
	={}& \frac{   \NPermuteR{k}{3} \NPermuteR{X_{\mathrm{SI}}}{3}  }{\NPermuteR{X_{\mathrm{S} \Bigcdot }}{3} }  \eqstop
\end{align*}
Similarly, we can derive for $i \in S_k$,
\begin{align*}
	\Eof{    \NPermuteR{X_{\mathrm{SI},i}}{2} \mid \history{t} } ={} & \sum_{n_{\mathrm{SI}} } \NPermuteR{n_{\mathrm{SI}}}{2} \frac{  \binom{X_{\mathrm{SI}} }{n_{\mathrm{SI}}}   \binom{ X_{\mathrm{S} \Bigcdot } - X_{\mathrm{SI}}  }{k-n_{\mathrm{SI}}  } }{   \binom{X_{\mathrm{S} \Bigcdot } (t) }{k} }
	={} \frac{   \NPermuteR{k}{2} \NPermuteR{X_{\mathrm{SI}}}{2}  }{\NPermuteR{X_{\mathrm{S} \Bigcdot }}{2} }  \eqcomma
\end{align*}
whence we get
\begin{align*}
	\Eof{   X_{\mathrm{SI},i}^3 \mid \history{t} } ={} & \Eof{    \NPermuteR{X_{\mathrm{SI}, i}}{3} \mid \history{t} } + 3 \Eof{    \NPermuteR{X_{\mathrm{SI},i}}{2} \mid \history{t} } + \Eof{  X_{\mathrm{SI},i} \mid \history{t} } \\
	={} &  \frac{   \NPermuteR{k}{3} \NPermuteR{X_{\mathrm{SI}}}{3}  }{\NPermuteR{X_{\mathrm{S} \Bigcdot }}{3} }  + 3  \frac{   \NPermuteR{k}{2} \NPermuteR{X_{\mathrm{SI}}}{2}  }{\NPermuteR{X_{\mathrm{S} \Bigcdot }}{2} }  + k  \frac{ X_{\mathrm{SI}} }{X_{\mathrm{S} \Bigcdot } }  \eqstop
\end{align*}
Proceeding in a similar fashion, for $i \in S_k$,
\begin{align*}
	\Eof{  X_{\mathrm{SI},i}  \NPermuteR{X_{\mathrm{SS},i}}{2} \mid \history{t} }  ={} & \sum_{ 0 \leq n_{\mathrm{SI}}  + n_{\mathrm{SS}}=k } n_{\mathrm{SI}} \NPermuteR{n_{\mathrm{SS}}}{2} \frac{  \binom{X_{\mathrm{SI}} }{n_{\mathrm{SI}}}    \binom{ X_{\mathrm{S} \Bigcdot }  - X_{\mathrm{SI}}  }{n_{\mathrm{SS}} } }{   \binom{X_{\mathrm{S} \Bigcdot }  }{k} }   \\
	={}& \sum_{ 0 \leq n_{\mathrm{SI}}  + n_{\mathrm{SS}}=k   } n_{\mathrm{SI}} \NPermuteR{n_{\mathrm{SS}}}{2} \frac{    \frac{X_{\mathrm{SI}} }{  n_{\mathrm{SI}}  }             \binom{X_{\mathrm{SI}} -1 }{n_{\mathrm{SI}}-1 }     \frac{  \NPermuteR{X_{\mathrm{SS}}}{2}   }{  \NPermuteR{n_{\mathrm{SS}}}{2}}         \binom{ X_{\mathrm{S} \Bigcdot } - X_{\mathrm{SI}}  -2  }{k-n_{\mathrm{SI}} -2  } }{  \frac{ \NPermuteR{X_{\mathrm{S} \Bigcdot }}{3} }{  \NPermuteR{k}{3}  } \binom{X_{\mathrm{S} \Bigcdot } -3 }{k-3} }    \\
	={}& \frac{   \NPermuteR{k}{3} X_{\mathrm{SI}} \NPermuteR{X_{\mathrm{SS}}}{2}  }{\NPermuteR{X_{\mathrm{S} \Bigcdot }}{3} }  \sum_{n_{\mathrm{SI}}  }  \varsigma  (n_{\mathrm{SI}}, ; X_{\mathrm{SI}}-1, X_{\mathrm{S} \Bigcdot}-3, k-3   )  \\
	={} & \frac{   \NPermuteR{k}{3} X_{\mathrm{SI}}  \NPermuteR{X_{\mathrm{SS}}}{2}  }{\NPermuteR{X_{\mathrm{S} \Bigcdot }}{3} }  \eqstop
\end{align*}
Similarly, for $i \in S_k$,
\begin{align*}
	\Eof{   \NPermuteR{X_{\mathrm{SI},i}}{2} X_{\mathrm{SS},i} \mid \history{t}  }  ={} & \frac{   \NPermuteR{k}{3}   \NPermuteR{X_{\mathrm{SI}}}{2} X_{\mathrm{SS}}  }{\NPermuteR{X_{\mathrm{S} \Bigcdot }}{3} } \eqcomma \\
	\Eof{  X_{\mathrm{SI},i} X_{\mathrm{SS},i}  \mid \history{t} } ={} & \frac{   \NPermuteR{k}{2} X_{\mathrm{SI}} X_{\mathrm{SS}}  }{\NPermuteR{X_{\mathrm{S} \Bigcdot }}{2} } \eqstop
\end{align*}

\section{Interpretation of the $\diffOperator{}{}$ operator}
\label{sec:DOperator}

Here, we provide an intuitive explanation for the $\diffOperator{}{}$ operator defined in Equation~\eqref{eq:diffOperatorDefn} in the context of SI process on CM random graphs. Recall that $\mu_{\mathrm{S}}$, and $  \mu_{\mathrm{S}}^{(r)} $ denote the average degree of a randomly chosen susceptible vertex, and the average excess degree of a  susceptible vertex randomly chosen as a neighbour of $r$ infected individuals, respectively. In Section~\ref{sec:LLN}, we mentioned that the operator $\diffOperator{\psi\circ \vartheta  }{r+1} $  recursively compared a susceptible vertex randomly chosen as a neighbour of $r$ infected individuals with a randomly chosen susceptible vertex. We make this notion of comparison precise.

\begin{myLemma}
	\label{thm:D_recurrence}
	Assume \ref{itm:as1}, \ref{itm:as2}, and  \ref{itm:as3} for the stochastic SI model on configuration model graph $\mathcal{G}(\psi, n)$. Then, $  \diffOperator{\psi\circ\theta  }{r} \ConvInProb  \diffOperator{\psi\circ\vartheta  }{r}  $ uniformly on $\setTime{0}$, and the following  recurrence relation for $\diffOperator{}{r}$ holds
	\begin{align}
		 \diffOperator{\psi \circ \theta   }{r+1} =  \frac{ \mu_{\mathrm{S}}^{(r)}  (\theta)  }{  \mu_{\mathrm{S}} (\theta) }  \diffOperator{\psi\circ \theta  }{r} \eqstop
	\end{align}
\end{myLemma}

\begin{proof}[Proof of Lemma~\ref{thm:D_recurrence}]
	The probability that a randomly chosen vertex $i$ is susceptible and is of degree~$k$ is given by $\probOf{ i \in S_k(t)  } = n^{-1} X_{\mathrm{S}} (0)  \theta^k  (t) p_k  $. The following is then immediate.
	\begin{align*}
		\mu_{\mathrm{S}} (\theta(t)  ) = \sum_k k \probOf{   i \in S_k(t)  \mid  i \in S(t)   } = \frac{ \sum_{k}  k  \theta^k  (t) p_k}{  \sum_{k}  \theta^k (t) p_k  } =\frac{ \theta (t)  \partial \psi(\theta (t) )  }{ \psi(\theta(t) )  } \eqstop
	\end{align*}
In order to explicitly calculate $  \mu_{\mathrm{S}}^{(r)} $, it will be helpful to keep the dynamic construction of the graph in mind. In particular, we make use the neighbourhood distribution of a susceptible vertex given in Lemma~\ref{remark:susceptibleNbd}. 
Therefore,
\begin{align*}
	  \mu_{\mathrm{S}}^{(r)}(\theta(t)  ) ={} &   \frac{  \sum_k (k-r) \probOf{   i \in S_k(t)  } \Eof{   \NPermuteR{X_{\mathrm{SI},i}}{r }   \mid \history{t-}    }   }{   \sum_k  \probOf{   i \in S_k(t)  }  \Eof{   \NPermuteR{X_{\mathrm{SI},i}}{r }   \mid \history{t-}    }    } \\
		={} &  \frac{  \sum_k \NPermuteR{k}{r+1}   \theta^k (t) p_k   }{   \sum_k   \NPermuteR{k}{r}   \theta^k (t) p_k     } \\
		={} & \frac{ \theta (t) \,  \partial^{r+1} \psi(\theta (t) )  }{ \partial^{r} \psi(\theta (t) ) } \eqstop
\end{align*}

The recurrence relation then follows in a straightforward manner.
\begin{align*}
	 \diffOperator{\psi\circ\theta  }{r+1} =   \frac{ \theta \,  \partial^{r+1} \psi(\theta  )  }{ \partial^{r} \psi(\theta  ) } \times \frac{  \psi(\theta)  }{   \theta\partial \psi(\theta )  }  \times   \frac{   \psi^{r-1} (\theta )   \partial ^r \psi(\theta ) }{ (\partial \psi(\theta )  )^r } = \frac{ \mu_{\mathrm{S}}^{(r)}  (\theta)  }{  \mu_{\mathrm{S}} (\theta) }  \diffOperator{\psi\circ\theta  }{r} \eqstop
\end{align*}
The convergence $  \diffOperator{\psi\circ \theta  }{r} \ConvInProb  \diffOperator{\psi\circ\vartheta  }{r}  $, uniformly on $\setTime{0}$, follows virtue of Theorem~\ref{thm:LLN}. This completes the proof.
\end{proof}

For our purposes, we only need $ \frac{ \mu_{\mathrm{S}}^{(1)}  (\theta)  }{  \mu_{\mathrm{S}} (\theta) }  \ConvInProb  \kappa(\vartheta) =   \diffOperator{\psi\circ \vartheta }{2}$, $ \frac{ \mu_{\mathrm{S}}^{(2)}  (\theta)  }{  \mu_{\mathrm{S}} (\theta) }  \kappa(\theta) \ConvInProb  \diffOperator{\psi\circ \vartheta  }{3}  $, and hence the interpretation  in Section~\ref{sec:LLN} as a limiting ratio follows. The two operators
$\diffOperator{\psi \circ\vartheta}{2}$, and $\diffOperator{\psi \circ\vartheta}{3}$ essentially allow us to correctly estimate various pair and triple counts in the large graph limit.

 \section{Rebolledo's theorem}
 \label{sec:rebolledo}
Here, we furnish a statement of the Rebolledo theorem as our derivation of the functional central limit theorem relies heavily on it. The version presented here differs slightly from the original one in \cite{rebolledo1980central}. In our application, we only need a central limit theorem for locally square integrable martingales.  Therefore, we borrow the following version from \cite{andersen1997statistical}.

Let $M_1, M_2, \ldots, M_n, \ldots$ be a sequence of vector-valued  locally square integrable martingales. We allow the possibility of they being defined on different sample spaces for each $n$. Let us denote the components of $M_n$ as $M_{n} \defeq (M_{n,1}, M_{n,2}, \ldots, M_{n,k})$ for some $k \in \setOfNaturals$. Now, for each $\epsilon >0$, define $ M_n^{(\epsilon)}$ to be a vector-valued locally square integrable martingale that contains all jumps of $ M_n$ larger in absolute value than $\epsilon$ as only jumps. Write $M_n^{(\epsilon)} \defeq (M_{n,1}^{(\epsilon)}, M_{n,2}^{(\epsilon)}, \ldots, M_{n,k}^{(\epsilon)}) $. Therefore, the absolute value of the difference between jump sizes of $M_{n,i}$ and $M_{n,i}^{(\epsilon)}$ is necessarily smaller than $\epsilon$.

The optional quadratic (co-) variation process $\optionalVariation{M_{n,i}, M_{n,j}}$ of the processes $M_{n,i}$ and $M_{n,j}$ is defined as the sum of the product of the jump sizes of the processes. That is, for $i,j =1, 2, \ldots, k$,  we define
\begin{align*}
	\optionalVariation{M_{n,i}, M_{n,j}} \defeq \sum_{s\le t}  \delta M_{n,i}(s) \delta M_{n,j}(s) \eqstop 
\end{align*}
By convention, we define $\optionalVariation{M_{n,i}}  = \optionalVariation{M_{n,i}, M_{n,j}} $.  For $i, j = 1, 2, \ldots, k$, the predictable quadratic (co-) variation process between $M_{n,i}$ and $M_{n,j}$ is given by
\begin{align}
\predictableVariation{M_{n,i},M_{n,j}}(t) &{} \defeq \int_0^t \Cov{\differential{ M_{n,i} (s) }, \differential{ M_{n,j}(s)} \mid \history{s-}  }  \differential{s} \eqcomma \nonumber
\end{align}
with the convention $\predictableVariation{M_{n,i}} = \predictableVariation{M_{n,i},M_{n,i}} $. Here, $\differential{M_{n,i}(s)}$ and $\differential{M_{n,j}(s)}$ are the increments of $M_{n,i}$, and $M_{n,j}$ respectively. Denote the corresponding $k\times k$ matrix-valued quadratic variation processes as $\predictableVariation{M_n} \defeq (( \predictableVariation{M_{n,i},M_{n,j}}  ))$. Similarly, for the  process  $M_n^{(\epsilon)}$, denote the corresponding predictable quadratic variation process as  $\predictableVariation{M_n^{(\epsilon)}} \defeq (( \predictableVariation{M_{n,i}^{(\epsilon)},M_{n,j}^{(\epsilon)}}  ))$.  Similarly, let $\optionalVariation{M_n}$ denote the matrix of optional quadratic variations of the components of $M_n$. That is, $\optionalVariation{M_n} \defeq (( \optionalVariation{M_{n,i},M_{n,j}}  ))$. Finally, let $ \setTime{0} \defeq [0,T]$ be as before and let $\setTime{}\subset \setTime{0}$.
%

\begin{myTheorem}[Rebolledo's theorem for locally square integrable martingales]
 Consider the following three conditions
\begin{align}
  \predictableVariation{M_n}(t) \ConvInProb V(t), \forall t \in \setTime{} \label{eq:rebolledo_pred} \eqcomma  \\
  \optionalVariation{M_n}(t) \ConvInProb V(t), \forall t \in \setTime{} \label{eq:rebolledo_opt}  \eqcomma \\
  \predictableVariation{M_{n,i}^{(\epsilon)}}(t) \ConvInProb 0, \forall t \in \setTime{},  \epsilon>0, i=1,2,\ldots,k \eqcomma   \label{eq:lindeberg}
\end{align}
as $n\rightarrow\infty$, where $V$ is a $k\times k$ semidefinite matrix-valued, continuous deterministic   function on $\setTime{0}$   with positive semidefinite increments and $V(0)=\boldsymbol{0}$, the zero matrix.

Then, either of Equation~\eqref{eq:rebolledo_pred} or Equation~\eqref{eq:rebolledo_opt} together with Equation~\eqref{eq:lindeberg} imply the following finite-dimensional convergence:
\begin{align}
(M_n(t_1),M_n(t_2),\ldots,M_n(t_l)) \ConvInDist (W(t_1),W(t_2),\ldots,W(t_l)) \eqcomma \nonumber
\end{align}
as $n\rightarrow\infty$, where $W$ is a Gaussian vector martingale with $\optionalVariation{W} = \predictableVariation{W} = V$ and $t_1, t_2, \ldots, t_l \in \setTime{}$; moreover, both Equation~\eqref{eq:rebolledo_pred} and Equation~\eqref{eq:rebolledo_opt} then hold.

If, in addition, $\setTime{} $ is dense in $\setTime{0}$, then either of Equation~\eqref{eq:rebolledo_pred} or Equation~\eqref{eq:rebolledo_opt} together with Equation~\eqref{eq:lindeberg} imply the following weak convergence:
\begin{align*}
  M_n \ConvInDist W
\end{align*}
as  $n\rightarrow\infty$  in   $D^{(k)}$, the 
 space of $\setOfReals^k$-valued \cadlag functions on $\setTime{0}$  endowed with the Skorohod topology, and $\predictableVariation{M_n}$ and $\optionalVariation{M_n}$ converge uniformly on compact subsets of $\setTime{0}$ to $V$ in probability.
\end{myTheorem}

\acks
\noindent First of all, the authors would like to thank the anonymous reviewers for their constructive feedback, which significantly improved the quality of the paper. Most of the work was conducted when the first author was a PhD student at the Technische Universit\"at Darmstadt in Germany. It was initiated when the first author was visiting the Mathematical Biosciences Institute (MBI) at The Ohio State University in Summer 2016.  The first author acknowledges the hospitality of  MBI during his visit to the institute.

\fund 
\noindent The work  has been co-funded by the German Research Foundation~(DFG) as part of project C3 within the Collaborative Research Center~(CRC) 1053 -- MAKI and the National Science Foundation under grants  RAPID DMS-1513489 and DMS-1853587. MBI is receiving major funding from the National Science Foundation under  the grant  DMS-1440386. The first author was supported by the President's Postdoctoral Scholars Program (PPSP) of the Ohio State University.

\competing 
\noindent There were no competing interests to declare which arose during the preparation or publication process of this article.

\data 
\noindent Not applicable.


\end{document}